\documentclass{amsart}

\usepackage[utf8]{inputenc}
\usepackage[american]{babel}
\usepackage{amsmath,amssymb,amsthm}
\usepackage{csquotes}
\usepackage{braket}
\usepackage{comment}
\usepackage{enumerate}
\usepackage{dsfont}
\usepackage{microtype}
\usepackage[left=3cm,right=3cm,top=3cm,bottom=3cm]{geometry}
\usepackage[style=alphabetic,sorting=anyt,backend=biber,doi=false,url=false,maxbibnames=10,isbn=false]{biblatex}

\usepackage[hidelinks]{hyperref}
\usepackage{cleveref}
\usepackage{xurl}
\hypersetup{breaklinks=true}

\newcommand{\1}{\mathds{1}}
\newcommand{\IC}{\mathbb C}

\newcommand{\cJ}{\mathcal{J}}
\renewcommand{\L}{\mathcal L}

\newcommand{\IR}{\mathbb R}

\newcommand{\abs}[1]{\lvert#1\rvert}
\newcommand{\norm}[1]{\lVert#1\rVert}

\renewcommand{\epsilon}{\varepsilon}
\renewcommand{\phi}{\varphi}

\newcommand{\real}{\mathbb{R}}

\newcommand{\un}{\mathbf{1}}

\newcommand{\com}{\mathbb{C}}

\newcommand{\ran}{\operatorname{ran}}

\newcommand{\A}{A}
\newcommand{\B}{B}
\newcommand{\V}{V}

\newtheorem{theorem}{Theorem}[section]
\newtheorem{lemma}[theorem]{Lemma}
\newtheorem{proposition}[theorem]{Proposition}
\newtheorem{corollary}[theorem]{Corollary}

\theoremstyle{remark}

\newtheorem{remark}[theorem]{Remark}
\newtheorem{example}[theorem]{Example}

\theoremstyle{definition}
\newtheorem{definition}[theorem]{Definition}

\title{Intertwining Curvature Bounds for Graphs and Quantum Markov Semigroups}

\author{Florentin Münch}
\address{Max Planck Institute for Mathematics in the Sciences, Inselstraße 22, 04103 Leipzig, Germany}
\email{cfmuench@gmail.com}

\author{Melchior Wirth}
\address{Institute of Science and Technology Austria (ISTA), Am Campus 1, 3400 Klosterneuburg, Austria}
\email{melchior.wirth@ist.ac.at}

\author{Haonan Zhang}
\address{Department of Mathematics, University of South Carolina, Columbia, SC, 29208, USA}
\email{haonanzhangmath@gmail.com}

\begin{document}

\begin{abstract}
Based on earlier work by Carlen--Maas and the second- and third-named author, we introduce the notion of intertwining curvature lower bounds for graphs and quantum Markov semigroups. This curvature notion is stronger than both Bakry--Émery and entropic Ricci curvature, while also computationally simpler than the latter. We verify intertwining curvature bounds in a number of examples, including finite weighted graphs and graphs with Laplacians admitting nice mapping representations, as well as generalized dephasing semigroups and quantum Markov semigroups whose generators are formed by commuting jump operators. By improving on the best-known bounds for entropic curvature of depolarizing semigroups, we demonstrate that there can be a gap between the optimal intertwining and entropic curvature bound. In the case of qubits, this improved entropic curvature bound implies the modified logarithmic Sobolev inequality with optimal constant.
\end{abstract}

\maketitle

\section{Introduction}

Ricci curvature lower bounds have a plethora of applications in geometry and analysis on Riemannian manifolds (see the survey \cite{Wei07} for example). There are several characterizations of Ricci curvature lower bounds that extend beyond Riemannian manifolds, such as Bakry--Émery theory in terms of $\Gamma$-calculus of semigroups \cite{BE85}, and Lott--Sturm--Villani theory using geodesic convexity of the entropy in Wasserstein spaces \cite{Stu06a,Stu06b,LV09}, and recent years have seen a lot of activity in the study of Ricci curvature bounds for singular geometries.

There are however difficulties and divergent approaches when trying to extend these notions of Ricci curvature lower bounds to discrete spaces (graphs). On graphs, the Bakry--Émery Ricci curvature lower bound mimics the $\Gamma_2$-condition using the graph Laplacian \cite{LY10,Sch98}, while the entropic Ricci curvature lower bound is based on optimal transportation by introducing a dynamical Wasserstein distance \cite{Maa11,EM12,Mie13}. Unlike in the continuous setting, these two notions do not agree on graphs. Furthermore, there are many other, non-equivalent notions of Ricci curvature lower bounds on graphs, but we will mainly focus on the Bakry--Émery Ricci curvature lower bounds and entropic Ricci curvature lower bounds here. Both of them have their own advantages, for example, the former is usually easier to verify, while the latter is closely related to the gradient flow of the relative entropy so that from a (positive) Ricci curvature lower bound one may derive functional inequalities such as modified logarithmic Sobolev inequality. However, establishing entropic curvature lower bounds in concrete examples is often difficult.

These difficulties become even more challenging when moving to the noncommutative world where functions are replaced by operators and there may be no notion of underlying space. In this setting, instead of functions on a space one looks at an abstract algebra of operators. It includes both continuous and discrete spaces when the algebra is commutative. When the algebra is not commutative (like matrix algebras for example), which is often the case for plenty of interesting examples arising in operator algebras and mathematical physics, the above known theories and tools are not available. The study of Ricci curvature lower bounds in this context was motivated by problems arising in quantum information theory, noncommutative harmonic analysis and mathematical physics \cite{JM10,CM14,JZ15,CM17,MM17,DR20}. Among others, the two Ricci curvature lower bounds following Bakry--Émery and Lott--Sturm--Villani are still of particular interest in the noncommutative setting. On one hand, the study of these curvature conditions has seen great progress in the past decade inspired by work in the classical (continuous and discrete spaces) world. On the other hand, its development in this bigger picture also provides ideas and insights that are new in the classical framework, as we shall discuss in the present article. 

One contribution of this paper is to develop a method to partially overcome the aforementioned difficulty of proving entropic Ricci curvature lower bounds, either in the discrete or the noncommutative setting. It is inspired by certain intertwining identities which are elementary as well as crucial in some classical examples, such as the heat semigroup and the Ornstein--Uhlenbeck semigroup. In the noncommutative setting, this was first used by Carlen and Maas \cite{CM17} and was further developed by the second- and third-named author \cite{WZ21,WZ23}, Li, Junge and LaRacuente \cite{LJLR20} and Brannan, Gao and Junge \cite{BGJ22,BGJ23}. One may understand the intertwining approach as a linearization scheme that simplifies the heavy computations involved in obtaining entropic Ricci curvature lower bounds.

This idea has served so far as one of main tools in proving entropic Ricci curvature lower bounds in the noncommutative setting, but the flexibility provided by the formulation from \cite{WZ21} has barely been exploited so far, and in the case of graphs, the intertwining approach to entropic curvature bounds has not gotten attention yet.

Based on these ideas, we introduce the notion of intertwining curvature lower bounds for graphs and quantum Markov semigroups in this article. This curvature notion combines the advantages of Bakry--Émery and entropic curvature: It can be characterized by a $\Gamma_2$-type criterion, which makes it computationally accessible, and it has the same strong applications to functional inequalities as entropic curvature (in fact, it is stronger than both Bakry--Émery and entropic curvature lower bounds).

In general, a Bakry--Émery curvature lower bound does not imply any entropic curvature lower bound without further assumptions. One key finding of this article is that nonetheless for certain special classes of examples, an intertwining curvature bound and thus also an entropic curvature bound can be derived from a Bakry--Émery-type curvature bound.

Maybe surprisingly, such examples include discrete spaces. It seems to be novel to apply this intertwining technique to discrete examples. In particular, for finite weighted graphs, we give a universal intertwining curvature lower bound in terms of the vertex and edge weights. For some graphs with better structures, we recover known results in \cite{EM12} using this new intertwining method. In the noncommutative setting, we study quantum Markov semigroups with generators formed by commuting jump operators, quantum dephasing semigroups and depolarizing semigroups. Our results improve some known Ricci curvature lower bounds and sometimes give sharp estimates. In particular, we show that the optimal modified logarithmic Sobolev constant for the depolarizing semigroup on qubits can be obtained from an entropic curvature bound, thus answering a question of Carlen.

\subsection{Overview of graph Ricci curvature notions}

We now discuss how the intertwining curvature in the graph case fits in the wide landscape of discrete Ricci curvature notions. Besides the intertwining curvature, there are Bakry--\'Emery curvature \cite{LY10,Sch98}, entropic curvature \cite{EM12,Mie13}, Forman curvature \cite{For03,JM21} and Ollivier curvature \cite{Oll09,LLY11}.
Broadly speaking, there are two categories, namely $\ell^\infty/\ell^2$ type curvatures, and Hodge/graph Laplacian based curvatures, as shown in Table~\ref{tab:table1}, indicating that the intertwining curvature fills the gap between the other well established curvature notions.

\begin{table}[!htbp]
  \begin{center}
    \caption{Comparison of graph Ricci curvatures.}
    \label{tab:table1}
\begin{tabular}{c|c|c}
Category & $\ell^\infty$-gradient & $\ell^2$-gradient\\
\hline
Graph Laplacian &Ollivier curvature&  Bakry--\'Emery/Entropic curvature\\
\hline
Hodge Laplacian & Forman curvature& Intertwining curvature
\end{tabular}
  \end{center}
\end{table}
While the Ollivier curvature was originally defined via optimal transport, its connection to the graph Laplacian was established in \cite{MW19}.
For Ollivier and Forman curvature, the gradient is measured in an $\ell^\infty$-way, i.e., it is based on the Lipschitz constant with respect to the combinatorial metric. In contrast, for Bakry--\'Emery, entropic and intertwining curvature, the gradient is measured via some $\Gamma$-calculus which is based on a local scalar product for vector fields and thus $\ell^2$-like.

The Forman and intertwining curvature depend on a given Hodge Laplacian on vector fields. In both cases, there is the freedom of choosing the two-cells, given an underlying graph as 1-skeleton. When restricting to gradient vector fields, the Forman curvature turns into Ollivier curvature \cite{JM21}, and the intertwining curvature into the Bakry--\'Emery curvature as shown later in the article. Therefore, the Hodge Laplacian based curvature is in both cases a lower bound for the corresponding graph Laplacian based curvature.

The following table presents an overview of the known consequences of lower Ricci curvature bounds. As lower Forman curvature bounds have almost the same consequences as the weaker Ollivier curvature bounds, we omit the Forman curvature in the table.
\begin{table}[h!]
  \begin{center}
    \caption{Implications of suitable curvature lower bounds.}
    \label{tab:table2}
\begin{tabular}{c|c|c|c}
 & Bakry--\'Emery & Entropic & Ollivier\\
\hline
Lichnerowicz spectral gap &yes \cite{BCLL17}& yes \cite{EM12}& yes \cite{Oll09}  \\
\hline
Diameter bounds   &yes \cite{LMP18} &yes \cite{EM12,Kam20}& yes \cite{Oll09}\\
\hline
Gradient estimates & yes \cite{LL18}  &yes \cite{EM12,EF18} & yes \cite{MW19}\\
\hline
Gaussian concentration & yes \cite{Sch98}  &yes \cite{EM12} & yes \cite{JMR19}\\
\hline
Buser inequality & yes \cite{KKRT16}  &yes \cite{EF18} & yes \cite{Mun23}\\
\hline
No expanders & yes \cite{Sal22}  &yes? & yes \cite{Sal22}\\
\hline
Li--Yau/volume doubling & yes \cite{BHLLMY15,Mun19}  &? & ?
\\ \hline
Homology vanishing & yes \cite{KMY21,MR20}  &no? & yes?
\\ \hline
Log-Sobolev inequality & ?  &yes \cite{EM12,EF18} & yes \cite{Mun23b}
\\ \hline
Convexity of entropy & ?  &yes \cite{EM12} & ? 
\\
\end{tabular}
  \end{center}
\end{table}

As the intertwining curvature is a lower bound of both Bakry--\'Emery and entropic curvature, one obtains the implications from both worlds. 
Particularly, 
assuming positive or non-negative intertwining curvature, all properties listed in Table~\ref{tab:table2} follow immediately.

As the intertwining curvature for graphs is a very recent development, there are not many examples investigated yet.
It is therefore an interesting open question to give useful curvature bounds for graphs whose Bakry--\'Emery and entropic curvature bounds are well known such as interacting particle systems and random transpositions on the symmetric group \cite{KKRT16,EFS20,FM16,EMT15,EHMT17}.

Besides the discrete Ricci curvature notions discussed above, there was recently introduced a new version of entropic curvature based on Schrödinger bridges for the $\ell^1$ Wasserstein metric \cite{Sam22,RS23,Con22,Ped23}. This curvature notion can be seen as a crossover of Ollivier and entropic curvature by Erbar/Maas/Mielke.
Furhter Ricci curvature notions for graphs have been introduced recently in \cite{DL22} and \cite{Ste23}, however both seem to have no connection to Riemannian geometry. 

\subsection{Plan of the paper}
The plan of this paper is as follows. In Section \ref{sec:graphs} we focus on weighted graphs and Ricci curvature bounds based on the graph Laplacians. We study the family of gradient estimates $\mathrm{GE}_\Lambda(K,\infty)$ depending on a mean function $\Lambda$. In the case when $\Lambda$ is the arithmetic mean, this gradient estimate reduces to a Bakry--Émery curvature lower bound, while for $\Lambda$ the logarithmic mean, one obtains an entropic curvature lower bound. We then introduce intertwining curvature as a linearized version of the gradient estimate, which implies $\mathrm{GE}_\Lambda(K,\infty)$ for all means $\Lambda$ (\Cref{thm:intertwining_GE_graphs}).

We show that for certain examples, namely complete graphs (\Cref{ex:complete_graph}) and graphs whose graph Laplacian is given by a so-called mapping representation (\Cref{thm:mapping}), the optimal intertwining curvature bound coincides with the Bakry--Émery curvature bound. Furthermore, we give a universal lower bound on the intertwining curvature in terms of the minimal transition rate between vertices (\Cref{thm:universal_bound_intertwining}).

Section \ref{sec:QMS} then treats quantum Markov semigroups on matrix algebras. As in the case of graphs, we introduce a notion of intertwining curvature and show that it is stronger than both Bakry--Émery and entropic curvature (\Cref{thm:grad_est_from_intertwining_QMS}). We prove intertwining curvature lower bounds for dephasing semigroups in terms of the Pimsner--Popa index of the conditional expectation (\Cref{prop:intertwining_curv_dephasing}), which are sharp in the case of depolarizing semigroups. Furthermore we prove that if a family of quantum Markov semigroups has a uniform intertwining curvature lower bounds and their generators have mutually commuting jump operators, then the product of these semigroups has the same intertwining curvature bound (\Cref{thm:intertwining_QMS_very_commuting}).

In the last subsection, we focus on depolarizing semigroups and show that the entropic curvature lower bound obtained from intertwining curvature can be improved (\Cref{thm:GE_depolarizing}). Since the intertwining curvature bound is optimal, this yields the first noncommutative examples of entropic curvature bounds that cannot be obtained by intertwining. As an application, we conclude that the optimal modified logarithmic Sobolev inequality for the depolarizing semigroup on qubits can be obtained from an entropic curvature bound.

\subsection*{Acknowledgments}
M. W. and H. Z. want to thank Eric Carlen and Jan Maas for productive discussions. M. W. was funded by the Austrian Science Fund (FWF) under
the Esprit Programme [ESP 156]. For the purpose of Open Access, the authors have applied a CC BY public copyright licence to any Author Accepted Manuscript (AAM) version arising from this submission.




\section{Intertwining curvature for weighted graphs}\label{sec:graphs}

In this section we discuss two known curvature conditions for finite graphs, namely Bakry--Émery curvature \cite{LY10,Sch98} and entropic curvature \cite{Maa11,EM12,Mie13}, which belong to a larger class of curvature conditions that can be expressed by a common type of gradient estimates. Then we introduce a new curvature notion, \emph{intertwining curvature}, based on earlier intertwining approaches to entropic curvature bounds. This curvature notion is stronger than both entropic and Bakry--Émery curvature and in fact the whole class of gradient estimates mentioned above.

We show that for certain examples, namely complete graphs (\Cref{ex:complete_graph}) and graphs whose graph Laplacian is given by a so-called mapping representation (\Cref{thm:mapping}), the optimal intertwining curvature bound coincides with the Bakry--Émery curvature bound. Furthermore, we give a universal lower bound on the intertwining curvature in terms of the minimal transition rate between vertices (\Cref{thm:universal_bound_intertwining}).

Before we turn to curvature, let us shortly recap some basic facts and definition regarding analysis on finite graphs. In our presentation, we mostly follow \cite{KL12,KLW21}. One point of deviation is that we work with complex-valued functions. This difference is largely irrelevant in the present context, but working with complex-valued functions is standard for matrix algebras in the context of quantum information and operator algebras, hence we use complex scalars throughout the article.

A \emph{finite weighted graph} is a triple $(X,b,m)$ consisting of a finite set $X$, a symmetric function $b\colon X\times X\to [0,\infty)$ that vanishes on the diagonal and a function $m\colon X\to (0,\infty)$. Such a triple gives rise to the structure of a simple undirected graph with vertex set $X$ and edge set $\{\{x,y\}\mid x,y\in X,\,b(x,y)>0\}$.

We identify $m$ with the measure with counting density $m$ on $X$. Then $\ell^2(X,m)$ is the space of all complex-valued functions on $X$ endowed with the inner product
\begin{equation*}
\langle f,g\rangle_m=\sum_{x\in X}\overline{f(x)}g(x)m(x)
\end{equation*}
for $f,g\in\ell^2(X,m)$. 

The graph Laplacian of a finite weighted graph $(X,b,m)$ is the operator
\begin{equation*}
L\colon \ell^2(X,m)\to\ell^2(X,m),\,Lf(x)=\frac 1{m(x)}\sum_{y\in X}b(x,y)(f(x)-f(y)).
\end{equation*}
The graph Laplacian is a positive self-adjoint operator on $\ell^2(X,m)$. Note that some authors use the opposite sign convention that makes the graph Laplacian a negative self-adjoint operator. The associated heat semigroup $(P_t)_{t\geq 0}$ is given by $P_t=e^{-tL}$ for $t\geq 0$.

As with $m$, we identify $b$ with the measure with counting density $b$ on $X\times X$. We define the discrete gradient as
\begin{equation*}
\partial\colon \ell^2(X,m)\to\ell^2(X\times X,b/2),\,(\partial f)(x,y)=f(x)-f(y).
\end{equation*}

The following lemma shows that the graph Laplacian and the discrete gradient are related by an identity analogous to the relation $-\Delta=\nabla^\ast\nabla$ for the Laplacian and gradient on Euclidean space. Throughout this article, we always use $T^\ast :H_2\to H_1$ to denote the adjoint of a bounded linear operator $T:H_1\to H_2$, where $H_1$ and $H_2$ are two Hilbert spaces. 

\begin{lemma}
One has $L=\partial^\ast\partial$.
\end{lemma}
\begin{proof}
If $f,g\in \ell^2(X,m)$, then
\begin{align*}
\langle f,Lg\rangle_m&=\sum_{x\in X}\overline{f(x)}\sum_{y\in X}b(x,y)(g(x)-g(y))\\
&=\frac 1 2\sum_{x,y\in X}b(x,y)\overline{f(x)}(g(x)-g(y))+\frac 1 2\sum_{x,y\in X}b(x,y)\overline{f(y)}(g(y)-g(x))\\
&=\frac 1 2\sum_{x,y\in X}b(x,y)\overline{(f(x)-f(y))}(g(x)-g(y))\\
&=\langle\partial f,\partial g\rangle_{b/2}.\qedhere
\end{align*}
\end{proof}

Let us now turn to a notion of Ricci curvature lower bounds for finite weighted graph. For $f,g\colon X\to \IC$, we write $f\otimes g$ for the function on $X\times X$ given by $(f\otimes g)(x,y)=f(x)g(y)$.

Recall that a function $h:[0,\infty)\to [0,\infty)$ is \emph{operator monotone} if for any positive semi-definite matrices $A\ge B$ of arbitrary size, we have $f(A)\ge f(B)$. We say that a function $\Lambda\colon [0,\infty)\times[0,\infty)\to [0,\infty)$ is an \emph{operator mean function} if it is continuous and there exists an operator monotone function $h\colon (0,\infty)\to (0,\infty)$ such that $h(1)=1$ and $\Lambda(s,t)=sh(t/s)$ for $s,t>0$. See \cite{KA80} for more details.

\begin{definition}
Let $K\in \real$. Let $(X,b,m)$ be a finite weighted graph with graph Laplacian $L$ and heat semigroup $(P_t)$ and let $\Lambda\colon [0,\infty)\times[0,\infty)\to [0,\infty)$ be an operator mean function. We say that $(X,b,m)$ satisfies the gradient estimate $\mathrm{GE}_\Lambda(K,\infty)$ if
\begin{equation*}
\langle \partial(P_t f),\Lambda\circ (\rho\otimes\rho)\partial(P_t f)\rangle_{b/2}\leq e^{-2Kt}\langle \partial f,\Lambda\circ (P_t\rho\otimes P_t\rho)\partial f\rangle_{b/2}
\end{equation*}
for all $f\in \ell^2(X,m)$, $\rho\colon X\to [0,\infty)$ and $t\geq 0$.
\end{definition}

\begin{remark}
For the definition of the gradient estimate $\mathrm{GE}_\Lambda$, it is not essential that $\Lambda$ be an operator mean function. Indeed, a wider class of functions $\Lambda$ was considered for example in \cite{EM12}. However, when studying matrices in \Cref{sec:QMS}, it is important that $\Lambda$ is an operator mean function, so we chose to impose this condition also in this section for sake of uniformity. 
\end{remark}

For two choices of means $\Lambda$, the gradient estimate $\mathrm{GE}_\Lambda(K,\infty)$ coincides with well-known notions of discrete lower Ricci curvature bounds. Before we elaborate on this, we need another piece of notation. The \emph{carré du champ} or \emph{gradient form} associated with the graph Laplacian $L$ is given by
\begin{equation*}
\Gamma\colon\ell^2(X,m)\times\ell^2(X,m)\to\ell^1(X,m),\,\Gamma(f,g)=\frac 1 2(\overline{(Lf)}g+\overline{f}Lg-L(\overline f g)).
\end{equation*}
More explicitly,
\begin{equation*}
\Gamma(f,g)(x)=\frac 1{2m(x)}\sum_{y\in X}b(x,y)(\overline{f(x)-f(y)})(g(x)-g(y)).
\end{equation*}
We simply write $\Gamma(f)$ for $\Gamma(f,f)$.

\begin{example}
Let $\Lambda(s,t)=s$ for $s,t\geq 0$. This function is called the \emph{left-trivial mean}. For $f\in \ell^2(X,m)$ and $\rho\colon X\to [0,\infty)$ we have
\begin{align*}
\langle \partial f,\Lambda\circ(\rho\otimes\rho)\partial f\rangle_{b/2}&=\frac 1 2\sum_{x\in X}\rho(x)\sum_{y\in X}b(x,y)\abs{f(x)-f(y)}^2\\
&=\langle \rho,\Gamma(f)\rangle_m.
\end{align*}
In fact, since $\Gamma(f)=\Gamma(\bar f)$, this is true for the whole family of affine means $\Lambda(s,t)=\alpha s+(1-\alpha)t$ with $\alpha\in [0,1]$, in particular the arithmetic mean.

Thus the gradient estimate $\mathrm{GE}_\Lambda(K,\infty)$ is equivalent to
\begin{equation}\label{ineq:BE GE}
\langle \rho,\Gamma(P_t f)\rangle_m\leq e^{-2Kt}\langle P_t\rho,\Gamma(f)\rangle_m=e^{-2Kt}\langle\rho,P_t\Gamma(f)\rangle_m
\end{equation}
for $f\in\ell^2(X,m)$, $\rho\colon X\to (0,\infty)$, or simply
\begin{equation*}
\Gamma(P_t f)\leq e^{-2Kt}P_t\Gamma(f)
\end{equation*}
for $f\in\ell^2(X,m)$.

This is the exponential formulation of the well-known Bakry--Émery curvature condition for graphs \cite{LL18}. 
\end{example}

\begin{example}
Let $\Lambda(s,t)=\frac{s-t}{\log s-\log t}$ for $s,t>0$ and continuously extend it to $s=0$ or $t=0$. This function is called the \emph{logarithmic mean}. The gradient estimate $\mathrm{GE}_\Lambda(K,\infty)$ is equivalent to an entropic Ricci curvature lower bound $K$ in the sense Erbar--Maas \cite{EM12}, as was shown in \cite[Theorem 3.1]{EF18}.
\end{example}

The optimal constant $K$ for which $(X,b,m)$ satisfies $\mathrm{GE}_\Lambda(K,\infty)$ depends on the choice of operator mean $\Lambda$. In particular, a Bakry--Émery curvature lower bound $K$ does not imply an entropic curvature lower bound $K$, nor vice versa. This seems to be known among experts, but seems only implicit in the literature, so we provide explicit examples here.

\begin{example}
Let $X=\{0,1\}$, $b(0,1)=\lambda(1-\lambda)$ and $m(0)=\lambda$, $m(1)=1-\lambda$ with $\lambda\in (0,1)$. It was shown by Maas \cite[Remark 2.11, Proposition 2.12]{Maa11} that the entropic curvature is bounded below by
\begin{equation*}
K=\frac 1 2+\inf_{\beta\in (-1,1)}\frac 1{1-\beta^2}\frac{\lambda(1+\beta)-(1-\lambda)(1-\beta)}{\log(\lambda(1+\beta))-\log((1-\lambda)(1-\beta))}.
\end{equation*}
By the logarithmic--geometric mean inequality,
\begin{equation*}
K\geq \frac 1 2+\inf_{\beta\in (-1,1)}\frac 1 {1-\beta^2}\sqrt{\lambda(1-\lambda)(1-\beta^2)}=\frac 1 2+\sqrt{\lambda(1-\lambda)}.
\end{equation*}
In particular, if $\lambda\neq \frac 1 2$, then $K>\frac 1 2+\min\{\lambda,1-\lambda\}$.

On the other hand, the optimal bound in the Bakry--Émery criterion is $K=\frac 1 2+\min\{\lambda,1-\lambda\}$, see e.\,g. \Cref{ex:complete_graph} below.
\end{example}

\begin{example}
We take the example from \cite[Subsection 4.3]{Mun23b}, namely $X=\{1,2,3\}$, $b(1,2)=10$, $b(2,3)=1$, $b(1,3)=0$, $m_\epsilon(1)=1/\epsilon$, $m_\epsilon(2)=1$, $m_\epsilon(3)=1/20$.

This graph has Bakry--Émery curvature bounded below by $1$ independently from $\epsilon$, as can be computed via \cite[Proposition 2.1]{HM23}. However, as shown in \cite{Mun23b}, the optimal constant in the modified logarithmic Sobolev inequality goes to zero as $\epsilon\to 0$. Thus, $(X,b,m_\epsilon)$ cannot have a positive lower bound on the entropic curvature that is independent of $\epsilon$. In fact, as shown in \cite{KLMP23}, the entropic curvature lower bound goes to $-\infty$ as $\epsilon\to 0$.
\end{example}


What makes the Bakry--Émery curvature condition stand out among the family of gradient estimates $\mathrm{GE}_\Lambda$ is the affine dependence on $\rho$, which allows for it to be reformulated as a pointwise inequality for the carré du champ \eqref{ineq:BE GE}. In particular, since the inequality is quadratic in $f$, for fixed $x\in X$ the best possible constant $K$ such that $\Gamma_2(f)(x)\geq K\Gamma(f)(x)$ is valid for all $f\in \ell^2(X,m)$ is the smallest eigenvalue of a symmetric matrix. This makes the computation of Bakry--Émery curvature computationally much more accessible than the gradient estimate for other means such as the logarithmic mean \cite{CKLP22,CKLLS22}.

To deal with this non-affine dependence on $\rho$ in the case of non-trivial means, the second- and third-named author introduced \cite{WZ21} a sufficient criterion for the gradient estimate $\mathrm{GE}_\Lambda(K,\infty)$ using \emph{intertwining techniques} based on ideas pioneered by Carlen and Maas in the noncommutative setting \cite{CM17,CM20a}. We take this condition as definition for a new curvature notion for weighted graphs.

Let $\mathcal J\colon \ell^2(X\times X,b/2)\to \ell^2(X\times X,b/2),\,\mathcal J\xi(x,y)=-\overline{\xi(y,x)}$. In particular, $\mathcal J(\partial f)=\partial(\overline f)$ for $f\in \ell^2(X,m)$. Further let
\begin{equation*}
\vec\Gamma(\xi,\eta)(x)=\frac 1 {2m(x)} \sum_{y\in X}b(x,y)\overline{\xi(x,y)}\eta(x,y)
\end{equation*}
for $\xi,\eta\in \ell^2(X\times X,b/2)$ and $x\in X$. As usual, we use $\vec \Gamma (\xi)$ to denote $\vec \Gamma (\xi,\xi)$ for simplicity. Note that 
\begin{equation}\label{eq:vec gamma 1}
	\vec\Gamma(\partial f,\partial g)=\Gamma(f,g),\qquad f,g\in \ell^2(X,m)
\end{equation}
and 
\begin{equation}\label{eq:vec gamma 2}
	\langle f, \vec\Gamma(\xi,\eta)\rangle_{m}=\langle (f\otimes 1)\xi,\eta\rangle_{b/2}.
\end{equation}

\begin{definition}
Let $(X,b,m)$ be a finite weighted graph with heat semigroup $(P_t)$, and let $K\in\IR$. We say that $(X,b,m)$ has intertwining curvature lower bound $K$ if there exists a strongly continuous semigroup $(\vec P_t)$ on $\ell^2(X\times X,b/2)$ such that
\begin{itemize}
\item[(a)] $\partial P_t=\vec P_t\partial$ for all $t\geq 0$,
\item[(b)] $\vec P_t \cJ=\mathcal J\vec P_t$,
\item[(c)] $\vec \Gamma(\vec P_t \xi)\leq e^{-2Kt}P_t\vec\Gamma(\xi)$ for all $\xi\in \ell^2(X\times X,b/2)$ and $t\geq 0$.
\end{itemize}
\end{definition}

If we take $\xi=\partial f$ in condition (c), then the inequality reduces to the Bakry--Émery gradient estimate $\Gamma(P_t f)\leq e^{-2Kt}P_t\Gamma(f)$ thanks to (a). One key finding of \cite{WZ21} was that an intertwining curvature lower bound $K$ implies the gradient estimate $\mathrm{GE}_\Lambda(K,\infty)$ not only for the left-trivial mean, but all operator mean functions $\Lambda$ \cite[Theorem 2]{WZ21}.

We give a new perspective on this result here in terms of interpolation theory of Hilbert spaces. The key result we use is Donoghue's theorem \cite[Theorems 1,2]{Don67}. We only formulate it in the finite-dimensional case to avoid technicalities, but the result remains valid in infinite dimensions with appropriate modifications. If $H$ is a finite-dimensional Hilbert space and $A$ an invertible positive self-adjoint operator on $H$, then we let $H^A$ be the Hilbert space with underlying vector space $H$ and inner product $\langle \xi,\eta\rangle_A=\langle\xi,A\eta\rangle_H$.

Now let $H_1$, $H_2$ be finite-dimensional Hilbert spaces and $A_j$, $B_j$ commuting invertible positive self-adjoint operators on $H_j$, $j\in \{1,2\}$. If $\Lambda$ is an operator mean functions, then $(H_1^{\Lambda(A_1,B_1)},H_2^{\Lambda(A_2,B_2)})$ is an exact interpolation couple for $(H_1^{A_1},H_1^{B_1})$ and $(H_2^{A_2},H_2^{B_2})$. This means that if a linear operator $T\colon H_1\to H_2$ is a contraction from $H_1^{A_1}$ to $H_2^{A_2}$ and from $H_1^{B_1}$ to $H_2^{B_2}$, then it is also a contraction from $H_1^{\Lambda(A_1,B_1)}$ to $H_2^{\Lambda(A_2,B_2)}$.

\begin{theorem}\label{thm:intertwining_GE_graphs}
If the finite weighted graph $(X,b,m)$ has intertwining curvature lower bounded by $K$, then it satisfies the gradient estimate $\mathrm{GE}_\Lambda(K,\infty)$ for every operator mean function $\Lambda$.
\end{theorem}
\begin{proof}
Let $(\vec P_t)$ be a strongly continuous semigroup on $\ell^2(X\times X,b/2)$ realizing the intertwining curvature bound $K$. Property (c) implies
\begin{align*}
	\langle\rho, \vec \Gamma (\vec P_t \xi)\rangle_{m}\le e^{-2Kt} \langle \rho, P_t\vec \Gamma (\xi)\rangle_m,\qquad \rho: X\to [0,\infty)
\end{align*}
which is nothing but 
\begin{equation*}
	\langle \vec P_t\xi,(\rho\otimes 1)\vec P_t\xi\rangle_{b/2}\leq e^{-2Kt}\langle\xi,(P_t\rho\otimes 1)\xi\rangle_{b/2}.
\end{equation*}
Replacing $\xi$ with $\cJ \xi$, one obtains 
\begin{equation*}
	\langle \vec P_t \xi,(1\otimes \rho)\vec P_t\xi\rangle_{b/2}\leq e^{-2Kt}\langle \xi,(1\otimes P_t\rho)\xi\rangle_{b/2}.
\end{equation*}
It follows from the interpolation theorem quoted above that
\begin{equation*}
\langle \vec P_t\xi,\Lambda\circ (\rho\otimes \rho)\vec P_t\xi\rangle_{b/2}\leq e^{-2Kt}\langle \xi,\Lambda\circ(P_t\rho\otimes P_t\rho)\xi\rangle_{b/2}
\end{equation*}
for every operator mean function $\Lambda$.

Restricting to $\xi=\partial f$ and taking condition (a) into account, the gradient estimate $\mathrm{GE}_\Lambda(K,\infty)$ follows.
\end{proof}

\begin{remark}
The proof makes it transparent why we need to require condition (c) in the definition of intertwining curvature lower bounds for all $\xi\in \ell^2(X\times X,b/2)$ instead of only $\xi\in\operatorname{ran}\partial$: Unless $b=0$, the latter is not invariant under multiplication by $1\otimes\rho$ and $\rho\otimes 1$, which is a requirement in the interpolation theorem we use. Indeed, if $b(x,y)>0$, then $\1_x\partial(\1_x)(x,y)=1$ and $\1_x\partial(\1_x)(y,x)=0$ imply that $\1_x\partial(\1_x)\notin\operatorname{ran}\partial$.
\end{remark}

Like Bakry--Émery curvature, intertwining curvature has a reformulation akin to the $\Gamma_2$ criterion.

\begin{proposition}\label{prop:Gamma_2_intertwining}
A finite weighted graph $(X,b,m)$ with graph Laplacian $L$ has intertwining curvature lower bound $K$ if and only if there exists a linear operator $\vec L$ on $\ell^2(X\times X,b/2)$ such that $\vec L\partial=\partial L$, $\vec L\mathcal J=\mathcal J L$ and
\begin{equation*}
\frac 1 2 (\vec \Gamma(\vec L\xi,\xi)+\vec\Gamma(\xi,\vec L\xi)-L\vec\Gamma(\xi))\geq K\vec\Gamma(\xi)
\end{equation*}
for all $\xi\in \ell^2(X\times X,b/2)$.
\end{proposition}
\begin{proof}
This follows from the standard semigroup arguments (compare with \cite[Theorem 3.1]{EF18} or \cite{LL18}.
\end{proof}

\begin{remark}
If we fix an operator $\vec L$ on $\ell^2(X\times X,b/2)$ that satisfies $\vec L\partial=\partial L$ and $\vec L \cJ=\cJ\vec L$, for every $x\in X$ the computation of the best possible $K\in\IR$ such that $\frac 1 2 (\vec \Gamma(\vec L\xi,\xi)+\vec\Gamma(\xi,\vec L\xi)-L\vec\Gamma(\xi))(x)\geq K\vec\Gamma(\xi)(x)$ for all $\xi\in \ell^2(X\times X,b/2)$ is a quadratic problem in $\xi$, which reduces to the computation of the smallest eigenvalue of a symmetric matrix, just like for Bakry--Émery curvature. For example, one may choose for $\vec L$ the idle Hodge Laplacian plus a multiple of the projection onto $(\ran\partial)^\perp$ (compare with Subsection \ref{subsec:idle_Hodge}). 
In contrast, the computation of the entropic curvature is a non-linear optimization problem, potentially with a multitude of local extrema and singularities which is a reason why it is notoriously hard to find good entropic curvature bounds for graphs with not too much symmetry.
As the intertwining curvature is a general lower bound to the entropic curvature, we expect that its comparably easy computation will be useful to derive new entropic curvature bounds for larger classes of graphs. 
\end{remark}

In some instances, a clever choice of the operator $\vec L$ reduces the computation of an intertwining curvature lower bound to the computation of the Bakry--Émery curvature bound. We will illustrate this with two examples, complete graphs (\Cref{ex:complete_graph}) and graphs whose graph Laplacian is given by a mapping representation (Subsection \ref{subsec:mapping_rep}).

\begin{example}\label{ex:complete_graph}
Let $X$ be a finite set, $m\colon X\to (0,\infty)$ with $\sum_x m(x)=1$ and $b(x,y)=m(x)m(y)$. The graph Laplacian associated with $(X,b,m)$ is given by
\begin{equation*}
Lf(x)=\frac 1 {m(x)}\sum_y b(x,y)(f(x)-f(y))=f(x)-\langle 1,f\rangle_m.
\end{equation*}
If $m$ is the uniform probability measure on $X$, then $L$ is the generator of the simple random walk on the complete graph over $X$.

Noting that all functions of the form $\partial f$ form a subspace of $\ell^2(X\times X,b/2)$, we can write any $\xi\in \ell^2(X\times X,b/2)$ as $\xi=\partial f+\eta$ with $\langle \partial f,\eta\rangle_{b/2}=0$ and define
\begin{equation*}
\vec L\xi=\partial Lf+2K\eta
\end{equation*}
with $K$ the Bakry--Émery curvature lower bound of $(X,b,m)$. Note that we may choose $K\geq \frac 1 2$. In fact, by definition of $L$ we have
\begin{align*}
	\Gamma(f,Lf)&=\Gamma(Lf,f)=\Gamma(f),\\
	L\Gamma(f)&=\Gamma(f)-\langle 1,\Gamma(f)\rangle_m.
\end{align*}
So 
\begin{equation}\label{ineq:BE 1/2}
2\Gamma_2(f)=\Gamma(f,Lf)+\Gamma(Lf,f)-L\Gamma(f)=2\Gamma(f)-\Gamma(f)+\langle1,\Gamma(f)\rangle_m\geq \Gamma(f).
\end{equation}

Then according to \eqref{eq:vec gamma 1} and \eqref{eq:vec gamma 2}, and by definitions of $L$ and $\vec{L}$, 
\begin{align*}
\vec \Gamma(\vec L\xi,\xi)&=\vec\Gamma(\partial Lf,\partial f)+\vec\Gamma(\partial L f,\eta)+2K\vec\Gamma(\eta,\partial f)+2K \vec\Gamma(\eta)\\
&=\Gamma(L f,f)+\vec\Gamma(\partial f,\eta)+2K\vec\Gamma(\eta,\partial f)+2K \vec\Gamma(\eta),\\
\vec \Gamma(\xi,\vec L\xi)&=\Gamma( f,L f)+\vec\Gamma(\eta,\partial f)+2K\vec\Gamma(\partial f,\eta)+2K \vec\Gamma(\eta),\\
L\vec\Gamma(\xi)&=L\vec\Gamma(\partial f,\partial f)+L\vec\Gamma(\partial f,\eta)+L\vec\Gamma(\eta,\partial f)+L\vec\Gamma(\eta)\\
&=L\Gamma(f)+\vec\Gamma(\partial f,\eta)-\langle 1, \Gamma(\partial f,\eta)\rangle_{m}+\vec\Gamma(\eta,\partial f)-\langle 1,\Gamma(\eta,\partial f)\rangle_m+\vec\Gamma(\eta)-\langle 1, \vec\Gamma(\eta)\rangle_m\\
&=L\Gamma(f)+\vec\Gamma(\partial f,\eta)+\vec\Gamma(\eta,\partial f)+\vec\Gamma(\eta)-\langle \partial f,\eta\rangle_{b/2}-\langle \eta,\partial f\rangle_{b/2}-\langle \eta, \eta\rangle_{b/2},\\
&=L\Gamma(f)+\vec\Gamma(\partial f,\eta)+\vec\Gamma(\eta,\partial f)+\vec\Gamma(\eta)-\langle \eta, \eta\rangle_{b/2}\\
\vec\Gamma(\xi)&=\vec\Gamma(\partial f,\partial f)+\vec\Gamma(\partial f,\eta)+\vec\Gamma(\eta,\partial f)+\vec\Gamma(\eta)\\
&=\Gamma(f)+\vec\Gamma(\partial f,\eta)+\vec\Gamma(\eta,\partial f)+\vec\Gamma(\eta).
\end{align*}
Therefore
\begin{align*}
	\vec\Gamma(\vec L\xi,\xi)+\vec\Gamma(\xi,\vec L\xi)-L\vec\Gamma(\xi)-2K \vec\Gamma(\xi)&=2(\Gamma_2(f)-K\Gamma(f))+(2K-1)\vec\Gamma(\eta)+\langle \eta, \eta\rangle_{b/2}\geq 0
\end{align*}
and $(X,b,m)$ has intertwining curvature lower bound $K$.

Let us give a more explicit lower bound for the Bakry--Émery curvature. Recalling \eqref{ineq:BE 1/2}, $(X,b,m)$ has Bakry--Émery lower bound $K$ if and only if 
\begin{equation}\label{ineq:bound of BE complete graph}
	\langle1,\Gamma(f)\rangle_m\geq (2K-1)\Gamma(f)(x),\qquad x\in X.
\end{equation}

We have 
\begin{align*}
\langle1,\Gamma(f)\rangle_m&=\frac 1 2\sum_{y,z\in X}m(y)m(z)\abs{f(y)-f(z)}^2\\
&\geq \frac 1 2\sum_{y\in X}m(y)m(x)\abs{f(y)-f(x)}^2+\frac 1 2\sum_{z\in X}m(x)m(z)\abs{f(x)-f(z)}^2\\
&=2m(x)\Gamma(f)(x)
\end{align*}
for every $x\in X$.

Thus we can take $K=\frac 1 2+\inf_{x\in X} m(x)$. In particular, if $m$ is the uniform probability measure on $\{1,\dots,n\}$, we obtain the intertwining curvature lower bound $K=\frac 1 2 +\frac 1 n$. For the entropic curvature, this bound has been obtained before by Mielke \cite[Example 3.4]{Mie13}.

The constant $K=\frac 1 2+\inf_{x\in X}m(x)$ is the optimal bound in the Bakry--Émery criterion and thus also the optimal intertwining curvature lower bound. Indeed, if $f=\mathds 1_x$, then $\langle1,\Gamma(f)\rangle_m=m(x)(1-m(x))$ and $\Gamma(f)(x)=\frac 1 2(1-m(x))$.
\end{example}

\subsection{Idle Hodge Laplacian and universal lower curvature bound}\label{subsec:idle_Hodge}

We first give a canonical choice of $\vec L$ corresponding to the Hodge Laplacian of the one-dimensional cell complex arising from the graph.
We then show that this choice gives a universal lower curvature bound.

Let $ \xi \in \ell^2(X\times X,b/2)$. We define 
\begin{align*}
\vec L\xi(x,y) := -\sum_z \left(\frac{b(y,z)}{m(y)}\xi(y,z) + \frac{b(x,z)}{m(x)} \xi(z,x) \right)
\end{align*}
For convenience, we will omit the large parentheses in the future.

We observe for $f \in \ell^2(X,m)$,
\[
\partial L f(x,y) = Lf(x) - Lf(y)
\]
and
\[
\vec L \partial f(x,y) = -\sum_z\left[ \frac{b(y,z)}{m(y)}(f(y)-f(z)) + \frac{b(x,z)}{m(x)} (f(z)-f(x)) \right]= Lf(x) - Lf(y).
\]
Moreover,
\[
\cJ \vec L \xi(x,y) = -\vec L \xi(y,x) = \sum_z \frac{b(x,z)}{m(x)}\xi(x,z) + \frac{b(y,z)}{m(y)} \xi(z,y)
\]
and
\[
 \vec L \cJ \xi(x,y) = \sum_z \frac{b(y,z)}{m(y)}\xi(z,y) + \frac{b(x,z)}{m(x)} \xi(x,z)
\]
giving $\cJ \vec L  = \vec L \cJ$.

We now estimate $\vec \Gamma(\xi, \vec L \xi)$. Notice first that to check $\vec \Gamma(\vec P_t \xi)\leq e^{-2Kt}P_t\vec\Gamma(\xi)$,
it is sufficient to check real-valued $\xi$, namely if $\xi= a+ ib$ for real-valued $a,b$, then
\[
\vec \Gamma(\vec P_t \xi)=\vec \Gamma(\vec P_t a) + \vec \Gamma(\vec P_t b)
\] 
as $\vec P_t$ maps real-valued functions to real-valued functions for our specific $\vec L$.

We use the notation $P(x,y) := b(x,y)/m(x)$. For real-valued $\xi$, we estimate 
\begin{align*}
2\vec \Gamma(\xi, \vec L \xi)(x) &= \sum_y P(x,y)\xi(x,y)\vec L \xi(x,y) \\
&= -\sum_{y,z}  P(x,y) \xi(x,y)\left( P(y,z)\xi(y,z) + P(x,z) \xi(z,x)     \right)\\
&= -\sum_{y,z}  P(x,y)P(y,z) \cdot \xi(x,y) \xi(y,z) + P(x,y)P(x,z)\cdot \xi(x,y)\xi(z,x) \\
&\ge - \sum_{y,z}  P(x,y)P(y,z)  \left(\xi(x,y)^2 + \frac 1 4 \xi(y,z)^2 \right) \\
&\qquad+ P(x,y)P(x,z)\left( \frac{\xi(x,y)^2}{P(z,x)}+\frac{1}{4}P(z,x)\xi(z,x)^2\right)\\
&= - \sum_{y,z}  P(x,y)P(y,z)  \xi(x,y)^2 + P(x,y)P(x,z) \frac{\xi(x,y)^2}{P(z,x)} \\
&\qquad+ \frac 1 4 P(x,y)P(y,z)\xi(y,z)^2 +\frac{1}{4}P(x,y)P(x,z)P(z,x)\xi(z,x)^2
\end{align*}
where we use the convention that $P(x,z)/P(z,x)=0$ if $P(x,z)=0=P(z,x)$.
By definition, we have 
\begin{align*}
	-L \vec \Gamma \xi(x) = \frac 1 2\sum_{y,z} P(x,y)P(y,z) \xi(y,z)^2 - P(x,y)P(x,z) \xi(x,z)^2  
\end{align*}
and 
\begin{align*}
	\vec\Gamma \xi (x)=\frac{1}{2}\sum_{y} P(x,y)\xi(x,y)^2.
\end{align*}


Now assume $\deg(x):=\sum_y P(x,y) \leq 1$. Then we have the estimates
\begin{align*}
	\sum_{y,z}  P(x,y)P(y,z)  \xi(x,y)^2
	\le 	\sum_{y}  P(x,y) \xi(x,y)^2
	=2\vec\Gamma \xi (x),
\end{align*}
\begin{align*}
	\sum_{y,z} P(x,y)P(x,z)P(z,x)\xi(z,x)^2
	& \leq \sum_z P(x,z)P(z,x)\xi(z,x)^2 \\
		&\leq \sum_{y,z} P(x,z)P(z,y)\xi(z,y)^2\\
&	=\sum_{y,z} P(x,y)P(y,z)\xi(y,z)^2,
\end{align*}
and 
\begin{align*}
\sum_{y,z}P(x,y)P(x,z) \xi(x,z)^2  
	\le \sum_{z}P(x,z) \xi(x,z)^2  
	=2\vec\Gamma (\xi)(x),
\end{align*}
so that 
\begin{align*}
	2\vec \Gamma(\xi, \vec L \xi)(x) 
	&\ge  - 2\vec\Gamma \xi (x) - \sum_{y,z}P(x,y)P(x,z) \frac{\xi(x,y)^2}{P(z,x)} - \frac{1}{2}\sum_{y,z} P(x,y)P(y,z)\xi(y,z)^2
\end{align*}
and
\begin{align*}
		-L \vec \Gamma \xi(x) 
		\ge  \frac 1 2\sum_{y,z} P(x,y)P(y,z) \xi(y,z)^2 - \vec\Gamma (\xi)(x).
\end{align*}
All combined, we obtain 
\begin{align*}
	2\vec \Gamma(\xi, \vec L \xi)(x)-L \vec \Gamma \xi(x)
	\ge - \sum_{y,z}P(x,y)P(x,z) \frac{\xi(x,y)^2}{P(z,x)} -3\vec\Gamma (\xi)(x).
\end{align*}

We further assume that $P(x,y)\ge 1/\alpha$ for some $\alpha>0$ and for all $x\sim y$. Then 
\begin{align*}
	 \sum_{y,z}P(x,y)P(x,z) \frac{\xi(x,y)^2}{P(z,x)} 
	 \le \alpha	 \sum_{y,z}P(x,y)P(x,z) \xi(x,y)^2 
	 \le  \alpha	 \sum_{y}P(x,y) \xi(x,y)^2 
	 =2\alpha\vec\Gamma (\xi)(x).
\end{align*}
Therefore, we conclude with 
\begin{align*}
	2\vec \Gamma(\xi, \vec L \xi)(x)-L \vec \Gamma \xi(x)
	\ge -(3+2\alpha)\vec\Gamma (\xi)(x),
\end{align*}
and hence, the curvature lower bound $K$ satisfies
\[
K\geq -\frac 3 2 - \alpha.
\]

We thus get the following Theorem.

\begin{theorem}\label{thm:universal_bound_intertwining}
Let $(X,b,m)$ be a graph. Assume $\sum_z P(x,z) \leq 1$ and $P(x,y) \geq P_{\min}>0$ for all $x\sim y$, where $P(x,y)=b(x,y)/m(x)$. Then, the intertwining curvature is lower bounded by
\[
K=-\frac 3 2 - \frac{1}{P_{\min}}.
\]
\end{theorem}

We notice that if a connected graph has only two vertices, then the curvature is positive. Otherwise, $P_{\min} \leq \frac 1 2$, which implies $\frac 3 2 \leq \frac 1 {P_{\min}}$. In either case,
\[
K \geq - \frac 2 {P_{\min}}.
\]

We remark that under the same assumptions, the Bakry--\'Emery curvature is lower bounded by $-1$. However no such estimate without employing $P_{\min}$ is true for the entropic curvature (see \cite{KLMP23}). However, in the same work, a similar bound as ours has been established for the entropic curvature.

For different choices of $\vec L$, the criterion for non-degenerate curvature seems to be

\[
\vec L \1_{x,y} (v,w) = 0 \mbox{ whenever } d(x,v) \geq 2
\]
as $\vec \Gamma (\xi, \vec L \xi)$ would otherwise contain a non-zero term $\xi(x,y)\xi(v,w)$ which cannot be compensated by the squares of $\xi$ coming from $L\vec \Gamma \xi$, as they only contain addends $\xi(v,w)^2$ with $d(x,v) \leq 1$. However, if the above criterion is met, a similar proof as above might go through, however the specific constants should be adapted. 

\begin{remark}
Similar to the case of Forman curvature \cite{For03,JM21},
there are various possible choices for the Hodge Laplacian for the intertwining curvature. Particularly, when maximizing the curvature, there should be triangles and squares included as 2-cells if possible as otherwise the curvature behaves like the curvature of a tree.
Indeed, $\vec L$ can be chosen as Hodge Laplacian for any cell complex, where there is an additional flexibility for $\vec L$ coming from the fact that vector fields are not necessarily anti-symmetric.

\end{remark}

\subsection{Graph Laplacians in terms of mapping representations}\label{subsec:mapping_rep}

We recall the notion of mapping representations following \cite{CPP09,EM12}. Let $(X,b,m)$ a finite weighted graph. A \emph{mapping representation} of the graph Laplacian of $(X,b,m)$ is a pair $(G,c)$ consisting of a set $G$ of maps from $X$ to itself and a function $c\colon X\times G\to [0,\infty)$ such that the following properties hold:
\begin{enumerate}
	\item [(a)] The graph Laplacian $L$ can be written as 
	\begin{equation*}
		L(f)(x)=\sum_{\delta \in G}c(x,\delta)\nabla_\delta(f)(x),\qquad \nabla_{\delta} f(x):=f(x)-f(\delta x).
	\end{equation*}
\item [(b)] For each $\delta \in G$ there exists a unique $\delta^{-1}\in G$ such that 
\begin{equation*}
	\delta^{-1}(\delta(x))=x,\qquad \textnormal{whenever}\qquad c(x,\delta)>0.
\end{equation*}
\item [(c)] For each $F:X\times G\to \real$ we have 
\begin{equation}\label{eq:dbc in mapping representation}
	\sum_{x\in X, \delta \in G}F(x,\delta)c(x,\delta)m(x)=\sum_{x\in X, \delta \in G}F(\delta x,\delta^{-1})c(x,\delta)m(x).
\end{equation}
\end{enumerate}
The condition (c) should be understood as a detailed balance condition. For ease of notation, we write $w(x,\delta)=c(x,\delta)m(x)$.

\begin{lemma}\label{lem:tangent_space_mapping_rep}
Let $(X,b,m)$ be a finite weighted graph and $(G,c)$ a mapping representation for the graph Laplacian of $(X,b,m)$. The linear map 
\begin{equation*}
V\colon \ell^2(X\times X,b/2)\to \ell^2(X\times G,w/2),\,\1_{(x,y)}\mapsto \sum_{\substack{\delta\in G\\\delta(x)=y}}\1_{(x,\delta)}
\end{equation*}
is an isometry.

Moreover
\begin{enumerate}[(a)]
\item $V\partial(f)(x,\delta)=\nabla_\delta f(x)$ for $f\in\ell^2(X,m)$, $x\in X$, $\delta\in G$,
\item $V((f\otimes g)\xi)(x,\delta)=f(x)g(\delta x)V\xi(x,\delta)$ for $f,g\in \ell^2(X,m)$, $\xi\in \ell^2(X\times X,b/2)$, $x\in X$, $\delta\in G$,
\item $V\cJ\xi(x,\delta)=-\overline{V\xi(\delta x,\delta^{-1})}$ for $\xi\in \ell^2(X\times X,b/2)$, $x\in X$, $\delta\in G$,
\item $\vec\Gamma(\xi,V^\ast \eta)(x)=\frac 1 2\sum_{\delta\in G}c(x,\delta)\overline{V\xi(x,\delta)}\eta(x,\delta)$ for $\xi\in \ell^2(X\times X,b/2)$, $\eta\in\ell^2(X\times G,w/2)$, $x\in X$.
\end{enumerate}
\end{lemma}
\begin{proof}
It follows from property (a) of the definition of mapping representations that
\begin{equation*}
b(x,y)=\sum_{\substack{\delta\in G\\\delta(x)=y}}c(x,\delta)m(x).
\end{equation*}
This readily implies that $V$ is an isometry.

(a) For $f\in \ell^2(X,m)$ we have
\begin{equation*}
V\partial(f)=\sum_{x,y\in X}(f(x)-f(y))V \1_{(x,y)}=\sum_{x,y\in X}\sum_{\substack{\delta\in G\\\delta(x)=y}}(f(x)-f(y))\1_{(x,\delta)}=\sum_{\substack{x\in X\\\delta\in G}}(f(x)-f(\delta x))\1_{(x,\delta)}.
\end{equation*}
(b) For $f,g\in \ell^2(X,m)$ and $\xi\in \ell^2(X\times X,b/2)$ we have
\begin{align*}
V((f\otimes g)\xi)&=\sum_{x,y\in X}f(x)g(y)\xi(x,y)V \1_{(x,y)}\\
&=\sum_{x,y\in X}f(x)g(y)\xi(x,y)\sum_{\substack{\delta\in G\\\delta(x)=y}}\1_{(x,\delta)}\\
&=\sum_{x\in X,\delta\in G}f(x)g(\delta(x))\xi(x,\delta(x)).
\end{align*}
(c) For $\xi\in \ell^2(X\times X,b/2)$ we have
\begin{equation*}
V\cJ\xi=-\sum_{x,y\in X}\overline{\xi(y,x)}V \1_{(x,y)}=-\sum_{x,y\in X}\overline{\xi(y,x)}\sum_{\substack{\delta\in G\\\delta(x)=y}}\1_{(x,\delta)}=-\sum_{x\in X,\delta\in G}\overline{\xi(\delta(x),x)}\1_{(x,\delta)}.
\end{equation*}
Hence $V\cJ\xi(x,\delta)=-\overline{\xi(\delta(x),x)}$. On the other hand,
\begin{align*}
V\xi(\delta(x),\delta^{-1})&=\sum_{y,z\in X}\sum_{\substack{\gamma\in G\\\gamma(y)=z}}\xi(y,z) \1_{\delta(x)}(y) \1_{\delta^{-1}}(\gamma)=\xi(\delta(x),x).
\end{align*}
(d) For $\xi\in \ell^2(X\times X,b/2)$, $\eta\in\ell^2(X\times G,w/2)$ and $\rho\in \ell^2(X,m)$ we have
\begin{align*}
	\langle\rho,\vec \Gamma(\xi,V^\ast \eta)\rangle_m
	=\langle (\rho\otimes 1)\xi,V^\ast\eta\rangle_{b/2}
	=\langle V((\rho\otimes 1)\xi),\eta\rangle_{w/2}.
\end{align*}
An application of (b) yields
\begin{equation*}
\langle V((\rho\otimes 1)\xi),\eta\rangle_{w/2}
=\langle (\rho\otimes 1)V\xi,\eta\rangle_{w/2}=\frac 1 2\sum_{x,\delta}c(x,\delta)\overline{V\xi(x,\delta)}\eta(x,\delta)\overline{\rho(x)}m(x).
\end{equation*}
Taking $\rho=\1_x/m(x)$ yields the claim.
\end{proof}

In particular, (a) implies that $L=\sum_{\delta\in G}\nabla_\delta^\ast\nabla_\delta$. In fact, the statement of the previous lemma can be seen as a special instance of the uniqueness of the first-order differential structure associated with a Markov generator (see also \Cref{prop:fodc} in the noncommutative case).

As a consequence of the previous lemma, for a finite weighted graph admitting a mapping representation $(G,c)$, the gradient estimate $\mathrm{GE}_\Lambda(K,\infty)$ relative to the operator mean $\Lambda$ can be formulated as 
\begin{equation*}
	\sum_{\delta \in G}\langle \nabla_\delta P_t f, \Lambda  (\rho, \rho_\delta)\nabla_\delta P_t f\rangle_{w/2}\le 
e^{-2Kt}\sum_{\delta \in G}\langle \nabla_\delta f, \Lambda (P_t\rho,(P_t\rho)_\delta)\nabla_\delta f\rangle_{w/2},
\end{equation*}
for all $f\in \ell^2(X,m),\rho:X\to [0,\infty),t\ge 0$. Here we use the convention that $\rho_\delta(x):=\rho(\delta x)$. When we choose $\Lambda$ to be the logarithmic mean, it becomes the entropic Ricci curvature lower bound.

By a slight abuse of notation, we write in the following $\cJ\xi(x,\delta)=-\overline{\xi(\delta x,\delta^{-1})}$, $\nabla f(x,\delta)=\nabla_\delta f(x)$ and $\vec \Gamma(\xi)=\vec \Gamma(\xi,\xi)$ with
\begin{equation*}
\vec \Gamma(\xi,\eta)=\frac{1}{2}\sum_{\delta\in G}\overline{\xi(x,\delta)}\eta(x,\delta)c(x,\delta).
\end{equation*}
Then the properties (a), (c) and (d) from the previous lemma can be succinctly expressed as $V\partial=\nabla$, $V\cJ=\cJ V$ and $\vec\Gamma(V\cdot\,,\cdot\,)=\vec \Gamma(\,\cdot\,,V^\ast\cdot\,)$.

With this notation, intertwining curvature bounds can be expressed in terms of mapping representations in a way analogous to \Cref{prop:Gamma_2_intertwining}.
\begin{proposition}\label{prop:gamma2 condition for intertwining}
A finite weighted graph $(X,b,m)$ with mapping representation $(G,c)$ has intertwining curvature lower bound $K$ if and only if there exists a linear operator $\vec L$ on $\ell^2(X\times G,w/2)$ such that $\vec L\nabla=\nabla L$, $\vec L\mathcal J=\mathcal J L$ and
\begin{equation}\label{eq:Gamma_2_mapping_rep}
\frac 1 2 (\vec \Gamma(\vec L\xi,\xi)+\vec\Gamma(\xi,\vec L\xi)-L\vec\Gamma(\xi))\geq K\vec\Gamma(\xi)
\end{equation}
for all $\xi\in \ell^2(X\times G,w/2)$.
\end{proposition}
\begin{proof}
First assume there exists an operator $\vec L$ on $\ell^2(X\times G,w/2)$ with the properties specified in the proposition and let $\vec L_b=V^\ast \vec L V$. By \Cref{lem:tangent_space_mapping_rep} (a), (c) we have $\vec L_b\partial=\partial L$ and $\vec L_b\cJ=\cJ\vec L_b$.

Moreover, by \Cref{lem:tangent_space_mapping_rep} (d), if $\xi\in \ell^2(X\times X,b/2)$, then
\begin{align*}
\vec\Gamma(\vec L_b\xi,\xi)+\vec\Gamma(\xi,\vec L_b\xi)-L\vec \Gamma(\xi)&=\vec \Gamma(V^\ast \vec L V\xi,\xi)+\vec \Gamma(\xi,V^\ast\vec L V\xi)-L\vec\Gamma(V^\ast V\xi,\xi)\\
&=\vec \Gamma(\vec L V\xi,V\xi)+\vec\Gamma(V\xi,\vec L V\xi)-L\vec\Gamma(V\xi)\\
&\geq 2K\vec\Gamma(V\xi)\\
&=2K\vec\Gamma(\xi).
\end{align*} 
Thus, by \Cref{prop:Gamma_2_intertwining}, $(X,b,m)$ has intertwining curvature bounded below by $K$.

Conversely assume that $(X,b,m)$ has intertwining curvature bounded below by $K$ and let $\vec L$ be an operator on $\ell^2(X\times X,b/2)$ that satisfies the conditions from \Cref{prop:Gamma_2_intertwining}. Let $\vec L_c=V\vec L V^\ast+K(1-VV^\ast)$. Again, it follows from \Cref{lem:tangent_space_mapping_rep} that $\vec L_c\nabla=\nabla L$ and $\cJ\vec L_c=\vec L_c \cJ$.

By \Cref{lem:tangent_space_mapping_rep} (d), $\vec \Gamma(VV^\ast\xi,(1-VV^\ast)\eta)=0$ for all $\xi,\eta\in \ell^2(X\times G,w/2)$. Thus it suffices to check \eqref{eq:Gamma_2_mapping_rep} for $\xi\in \operatorname{ran}V$ and $\xi\in (\operatorname{ran}V)^\perp$ separately. For $\xi\in \operatorname{ran}V$ it follows directly from the assumptions on $\vec L$, while the bound for $\xi\in (\operatorname{ran}V)^\perp$ is immediate from the definition of $\vec L_c$.
\end{proof}

Now we are ready to state the main result in this part. 

\begin{theorem}\label{thm:mapping}
Let $(X,b,m)$ be a finite weighted graph with mapping representation $(G,c)$. Consider the following conditions
\begin{enumerate}
\item [(a)] $\delta\circ\gamma=\gamma\circ\delta,$ for all $\delta, \gamma\in G$.
\item [(b)] $c(\delta x,\gamma)=c(x,\gamma)$, for all $x\in X$ and $\delta, \gamma\in G$.
\item [(c)] $\delta\circ \delta =\textnormal{id}$, for all $\delta \in G$.
\end{enumerate}
Then 
\begin{enumerate}
\item [(i)] if conditions (a) and (b) are satisfied, then we have the intertwining curvature bounded from below by $0$;
\item [(ii)] if conditions (a), (b) and (c) are satisfied, then we have the intertwining curvature bounded from below by $K$ with 
\begin{equation*}
K=2\min\left\{c(x,\delta): x\in X, \delta \in G \textnormal{ such that } c(x,\delta)>0\right\}.
\end{equation*}
\end{enumerate}
\end{theorem}

\begin{remark}
In particular, this implies the corresponding entropic Ricci curvature lower bounds obtained in \cite[Proposition 5.4]{EM12}.
\end{remark}

\begin{proof}
We prove (i) first. For this we define the linear operator $\vec L$ on $\ell^2(X\times G,w/2)$ by
\begin{equation*}
\vec L \xi (x,\delta)
=\sum_{\gamma\in G}c(x,\gamma)\left(\xi(x,\delta)-\xi(\gamma x,\delta)\right).
\end{equation*}
Then a direction computation shows 
\begin{equation*}
\vec L \nabla f(x,\delta)
= \sum_{\gamma\in G}c(x,\gamma)\left(f(x)-f(\delta x)-f(\gamma x)+f(\delta \gamma x)\right).
\end{equation*}
and 
\begin{equation*}
 \nabla L f(x,\delta)
=Lf(x)-Lf(\delta x)
=\sum_{\gamma\in G}(c(x,y)(f(x)-f(\gamma x))-c(\delta x,\gamma)(f(\delta x)-f(\gamma\delta x)))
\end{equation*}
By conditions (a) and (b), we have 
\begin{equation*}
\nabla L f(x,\delta)
=\sum_{\gamma\in G}(c(x,y)(f(x)-f(\gamma x)-f(\delta x)+f(\delta\gamma x))
=\vec L \nabla f(x,\delta).
\end{equation*}
So $\nabla L =\vec L \nabla$. 

By definition, we have
\begin{align*}
\vec L \cJ \xi (x,\delta)
&=\sum_{\gamma\in G}c(x,\gamma)(\cJ \xi(x,\delta)-\cJ \xi(\gamma x,\delta))=-\sum_{\gamma\in G}c(x,\gamma)(\overline{\xi (\delta x,\delta^{-1})}-\overline{\xi(\delta\gamma x,\delta^{-1})}),
\end{align*}
and 
\begin{align*}
 \cJ \vec L\xi (x,\delta)
=-\overline{\vec L \xi(\delta x,\delta^{-1})}=-\sum_{\gamma\in G} c(\delta x,\gamma)(\overline{\xi(\delta x,\delta^{-1})}-\overline{\xi(\gamma \delta x, \delta^{-1})}).
\end{align*}
Again, by conditions (a) and (b), we see that $\vec L \cJ= \cJ \vec L$.

According to Proposition \ref{prop:gamma2 condition for intertwining}, to finish the proof of (i), it remains to verify that
\begin{equation*}
\vec \Gamma(\vec L\xi,\xi)+\vec\Gamma(\xi,\vec L\xi)-L\vec\Gamma(\xi)\ge 0.
\end{equation*}
For this, one calculates 
\begin{align*}
\vec \Gamma(\vec L\xi,\xi)(x)
=\sum_{\gamma\in G}c(x,\gamma)\vec \Gamma(\xi)(x)-\frac{1}{2}\sum_{\delta,\gamma\in G}\overline{\xi(\gamma x,\delta)}\xi (x,\delta)c(x,\gamma)c(x,\delta)
\end{align*}
and
\begin{align*}
\vec \Gamma(\xi,\vec L\xi)(x)
=\sum_{\gamma\in G}c(x,\gamma)\vec \Gamma(\xi)(x)-\frac{1}{2}\sum_{\delta,\gamma\in G}\xi(\gamma x,\delta)\overline{\xi (x,\delta)}c(x,\gamma)c(x,\delta).
\end{align*}
Using condition (b), we find 
\begin{align*}
L \vec \Gamma(\xi)(x)
&=\sum_{\gamma\in G}c(x,\gamma)(\vec \Gamma(\xi)(x)-\vec \Gamma(\xi)(\gamma x))\\
&=\sum_{\gamma\in G}c(x,\gamma)\vec \Gamma(\xi)(x)-\frac{1}{2}\sum_{\delta,\gamma\in G}|\xi(\gamma x,\delta)|^2 c(x,\gamma)c(\gamma x,\delta)\\
&=\sum_{\gamma\in G}c(x,\gamma)\vec \Gamma(\xi)(x)-\frac{1}{2}\sum_{\delta,\gamma\in G}|\xi(\gamma x,\delta)|^2 c(x,\gamma)c(x,\delta).
\end{align*}
Therefore, 
\begin{align*}
&\;\vec \Gamma(\vec L\xi,\xi)(x)+\vec\Gamma(\xi,\vec L\xi)(x)-L\vec\Gamma(\xi)(x)\\
&\quad=\sum_{\gamma\in G}c(x,\gamma)\vec \Gamma(\xi)(x)+\frac{1}{2}\sum_{\delta,\gamma\in G}\left[|\xi(\gamma x,\delta)|^2-\overline{\xi(\gamma x,\delta)}\xi (x,\delta)-\xi(\gamma x,\delta)\overline{\xi (x,\delta)}\right]c(x,\gamma)c(x,\delta)\\
&\quad=\frac{1}{2}\sum_{\delta,\gamma\in G}\left[|\xi(x,\delta)|^2+|\xi(\gamma x,\delta)|^2-\overline{\xi(\gamma x,\delta)}\xi (x,\delta)-\xi(\gamma x,\delta)\overline{\xi (x,\delta)}\right]c(x,\gamma)c(x,\delta)\\
&\quad=\frac{1}{2}\sum_{\delta,\gamma\in G}|\xi(\gamma x,\delta)-\xi (x,\delta)|^2c(x,\gamma)c(x,\delta)\\
&\quad\ge 0,
\end{align*}
and we finish the proof of (i).

Now let us show (ii). For $\xi\in \ell^2(X\times G,w/2)$ let $\eta(x,\delta)=\frac 1 2(\xi(x,\delta)-\xi(\delta x,\delta))$ and $\zeta(x,\delta)=\frac 1 2(\xi(x,\delta)+\xi(\delta x,\delta))$ so that $\xi=\eta+\zeta$. Define
\begin{equation*}
\vec L_K \xi (x,\delta)
=\vec L\xi(x,\delta)+2K\zeta(x,\delta)
\end{equation*}
with $K$ as in the statement of the theorem and $\vec L$ as in (i).

If $\xi=\nabla f$ for some $f\in \ell^2(X,m)$, then $\zeta=0$ by (c), hence $\nabla L=\vec L_K\nabla$ follows from the intertwining identity $\nabla L=\vec L\nabla$ in (i). The identity $\vec L_K\cJ =\cJ \vec L_K$ follows from the same calculation as in the case (i) without assuming condition (c).

By Proposition \ref{prop:gamma2 condition for intertwining}, it reduces to proving 
\begin{equation*}
\vec \Gamma(\vec L_K\xi,\xi)+\vec\Gamma(\xi,\vec L_K\xi)-L\vec\Gamma(\xi)\geq 2K\vec\Gamma(\xi).
\end{equation*}

Using the computations from (i) we obtain
\begin{align*}
&\vec \Gamma(\vec L_K\xi,\xi)(x)+\vec\Gamma(\xi,\vec L_K\xi)(x)-L\vec\Gamma(\xi)(x)\\
&\quad= \Gamma(\vec L\xi,\xi)(x)+\vec\Gamma(\xi,\vec L\xi)(x)-L\vec\Gamma(\xi)(x)+2K\vec\Gamma(\zeta,\xi)+2K\vec\Gamma(\xi,\zeta)\\
&\quad=\frac{1}{2}\sum_{\delta,\gamma\in G}|\xi(\gamma x,\delta)-\xi (x,\delta)|^2c(x,\gamma)c(x,\delta)+2K\vec\Gamma(\zeta,\xi)(x)+2K\vec\Gamma(\xi,\zeta)(x)\\
&\quad\geq \frac 1 2\sum_{\delta\in G}\abs{\xi(\delta x,\delta)-\xi(x,\delta)}^2c(x,\delta)^2+2K\vec\Gamma(\zeta,\xi)(x)+2K\vec\Gamma(\xi,\zeta)(x)\\
&\quad=2\sum_{\delta\in G}\abs{\eta(x,\delta)}^2c(x,\delta)^2+2K\vec\Gamma(\zeta,\xi)(x)+2K\vec\Gamma(\xi,\zeta)(x)\\
&\quad\geq 2K(\vec\Gamma(\eta)(x)+\vec\Gamma(\zeta,\zeta)(x)+\vec\Gamma(\eta,\zeta)(x)+2\vec\Gamma(\zeta)(x))\\
&\quad\geq 2K\vec\Gamma(\eta+\zeta)(x)\\
&\quad=2K\vec\Gamma(\xi)(x).\qedhere
\end{align*}
\end{proof}

\section{Curvature conditions for quantum Markov semigroups on matrix algebras}\label{sec:QMS}

In this section we turn to quantum Markov semigroups on matrix algebras. The study of curvature in this context is more recent, with a significant uprise in activity after the construction of a noncommutative version of the Wasserstein distance by Carlen--Maas \cite{CM17} and Mittnenzweig--Mielke \cite{MM17} that made it possible to extend the Lott--Sturm--Villani approach to Ricci curvature lower bounds to the noncommutative setting. These methods have proven particularly useful in establishing exponential decay to equilibrium in relative entropy for open quantum systems (see also \cite{DR20,BGJ22,BGJ23}). In a different context, Bakry--Émery curvature bounds for quantum Markov semigroups (in infinite dimensions) have been studied with applications to noncommutative harmonic analysis \cite{JM10,JZ15}.

As in the case of graphs, we introduce a notion of intertwining curvature and show that it is stronger than both Bakry--Émery and entropic curvature (\Cref{thm:grad_est_from_intertwining_QMS}). We prove intertwining curvature lower bounds for dephasing semigroups in terms of the Pimsner--Popa index of the conditional expectation (\Cref{prop:intertwining_curv_dephasing}), which are sharp in the case of depolarizing semigroups. Furthermore we prove that if a family of quantum Markov semigroups has a uniform intertwining curvature lower bounds and their generators have commuting jump operators, then the product of these semigroups has the same intertwining curvature bound (\Cref{thm:intertwining_QMS_very_commuting}).

In the last subsection, we focus on depolarizing semigroups and show that the entropic curvature lower bound obtained from intertwining curvature can be improved (\Cref{thm:GE_depolarizing}). An an application, we conclude that the optimal modified logarithmic Sobolev inequality for the depolarizing semigroup on qubits can be obtained from an entropic curvature bound.

Let us briefly introduce the setup and fix notation. Let $M_n(\IC)$ denote the space of all complex $n\times n$ matrices. We write $\tau$ for the normalized trace on $M_n(\IC)$, that is, $\tau(\A)=\frac 1 n\sum_{j=1}^n \A_{jj}$. We use the Löwner order on $M_n(\IC)$, that is, $A\geq B$ if $A-B$ is positive, i.e. self-adjoint with spectrum contained in $[0,\infty)$. We use $\un$ to denote the identity matrices. 

A linear map $\Phi$ on $M_n(\IC)$ is called \emph{unital} if $\Phi(\un)=\un$ and \emph{completely positive} if 
\begin{equation*}
\sum_{j,k=1}^m B_j^\ast \Phi(A_j^\ast A_k)B_k\geq 0
\end{equation*}
for all $A_1,\dots,A_m,B_1,\dots,B_m\in M_n(\IC)$ and for all $m\ge 1$.

A quantum Markov semigroup on $M_n(\IC)$ is a strongly continuous semigroup of unital completely positive maps on $M_n(\IC)$. As a strongly continuous semigroup, a quantum Markov semigroup $(P_t)$ can be expressed as $P_t=e^{-t\L}$, where the operator $\L$ on $M_n(\IC)$ is uniquely determined by $(P_t)$ and called the \emph{generator of $(P_t)$}. As in the commutative case, we define the (iterated) carré du champ operator associated with $\L$ by
\begin{align*}
\Gamma(A,B)&=\frac 1 2\left(\L(A)^\ast B+A^\ast \L(A)-\L(A^\ast B)\right),\\
\Gamma_2(A,B)&=\frac 1 2\left(\Gamma(\L(A),B)+\Gamma(A,\L(B))-\L\Gamma(A,B)\right)
\end{align*}
for $A,B\in M_n(\IC)$. As usual, we write $\Gamma(A)$ for $\Gamma(A,A)$ and $\Gamma_2(A)$ for $\Gamma_2(A,A)$.

Given a linear map $T:M_n(\com)\to M_n(\com)$, we denote by $T^\dagger$ its adjoint with respect to the inner product induced by $\tau$, that is,
\begin{equation*}
\tau(T(A)B)=\tau(A T^\dagger(B))
\end{equation*}
for all $\A,\B\in M_n(\IC)$. 

In the following, we fix a positive definite matrix $\sigma\in M_n(\com)$ with $\tau(\sigma)=1$. We denote by $T^{\dagger, \sigma}$ the adjoint of $T$ with respect to the GNS inner product induced by $\sigma$, that is,
\begin{equation*}
\tau(T(A)B\sigma)=\tau(A T^{\dagger,\sigma}(B)\sigma)
\end{equation*}
for all $\A,\B\in M_n(\IC)$. A quantum Markov semigroup $(P_t)$ is called \emph{GNS-symmetric with respect to $\sigma$} if $P_t^{\dagger,\sigma}=P_t$ for all $t\ge 0$.

 Alicki \cite[Theorem 3]{Ali76} showed that a linear operator $\L$ on $M_n(\IC)$ is the generator of a quantum Markov semigroup on $M_n(\IC)$ that is GNS-symmetric with respect to $\sigma$ if and only if there exists a natural number $d$, non-zero operators $\V_j\in M_n(\IC)$ and numbers $\omega_j\in\IR$, $j\in \{1,\dots,d\}$ such that
\begin{itemize}
\item [(a)] $\tau(\V_j^\ast \V_k)=0$ for $j\neq k$,
\item [(b)] $\tau(\V_j)=0$ for $1\leq j\leq d$,
\item [(c)] for every $j\in \{1,\dots,d\}$ there exists $j^\ast\in \{1,\dots,d\}$ such that $\V_j^\ast=\V_{j^\ast}$,
\item [(d)] $\sigma \V_j \sigma^{-1}=e^{-\omega_j}\V_j$,
\item [(e)] the operator $\L$ acts as
\begin{equation*}
\L(\A)=\sum_{j=1}^d e^{-\omega_j/2}(\V_j^\ast[\V_j,\A]-[\V_j^\ast,\A]\V_j)
\end{equation*}
for $\A\in M_n(\IC)$.
\end{itemize}
Note that due to the orthogonality relation (a), the index $j^\ast$ in (c) is unique. 

In the case when $\sigma=\un$, a GNS-symmetric quantum Markov semigroup is also called \emph{tracially symmetric}. In this case, the dual semigroup $(P_t^\dagger)$ coincides with $(P_t)$. The generators of tracially symmetric quantum Markov semigroups are exactly the linear operators $\L$ on $M_n(\IC)$ that can be expressed as
\begin{equation*}
\L(\A)=\sum_{j=1}^d (\V_j^2 \A+\A\V_j^2-2\V_j \A \V_j)
\end{equation*}
with self-adjoint $\V_j\in M_n(\IC)$. 

While the numbers $\omega_j$'s are uniquely determined by $(P_t)$, the operators $\V_j$ are not. For this reason, it is not immediately clear if the gradient estimate for quantum Markov semigroups originally formulated by Carlen and Maas \cite{CM17} in terms of the operators $\V_j$ depends on this choice.

For this reason, we give an abstract formulation that relies on another representation result for the generators of GNS-symmetric quantum Markov semigroups by the second-named author \cite{Wir22a}. This result relies on the notions of \emph{Hilbert $C^\ast$-bimodules} and derivations, which we briefly recall.

A finite-dimensional Hilbert $C^\ast$-bimodule over $M_n(\IC)$ is a finite-dimensional bimodule $F$ over $M_n(\IC)$ together with a sesquilinear map $(\,\cdot\mid\cdot\,)\colon F\times F\to M_n(\IC)$ such that
\begin{itemize}
\item [(a)] $(\xi\vert \A\eta \B)=(\A^\ast\xi\vert\eta)\B$ for all $\xi,\eta\in F$, $\A,\B\in M_n(\IC)$,
\item [(b)] $(\xi\vert\xi)\geq 0$ for all $\xi\in F$,
\item [(c)] $(\xi\vert \xi)=0$ if and only if $\xi=0$.
\end{itemize}

By the representation theory of $M_n(\IC)$, for every finite-dimensional Hilbert $C^\ast$-bimodule $F$ over $M_n(\IC)$ there exists a natural number $d$ and a bijective bimodule map $\Phi\colon F\to M_n(\IC)^d$ such that $(\xi\vert \eta)=\frac 1 2\sum_{j=1}^d \Phi (\xi)_j^\ast \Phi(\eta)_j$.

If $F$ is a bimodule over $M_n(\IC)$, a linear map $\partial\colon M_n(\IC)\to F$ is called a \emph{derivation} if it satisfies the Leibniz rule $\partial(AB)=A\partial(B)+\partial(A)B$ for all $A,B\in M_n(\IC)$.

\begin{proposition}\label{prop:fodc}
If $\mathcal L$ generates a quantum Markov semigroup on $M_n(\IC)$ that is GNS-symmetric with respect to $\sigma$, then there exists a finite-dimensional Hilbert $C^\ast$-bimodule $F$ over $M_n(\IC)$, a strongly continuous group of isometries $(V_t)$ on $F$, an anti-linear operator $\cJ\colon F\to F$ and a derivation $\partial\colon M_n(\IC)\to F$ such that
\begin{enumerate}[(a)]
\item $V_t(\A\xi \B)=\sigma^{it}\A\sigma^{-it}(V_t\xi)\sigma^{it}\B\sigma^{-it}$ for all $\A,\B\in M_n(\IC)$, $\xi\in F$,
\item $\cJ(\A\xi \B)=\sigma^{1/2}\B^\ast\sigma^{-1/2}(\cJ \xi)\sigma^{1/2}\A^\ast\sigma^{-1/2}$ for all $\A,\B\in M_n(\IC)$, $\xi\in F$,
\item $\tau((\cJ\xi\vert\cJ\eta)\sigma)=\tau((\eta\vert\xi)\sigma)$ for all $\xi,\eta\in F$,
\item $\cJ V_t=V_t\cJ$ for all $t\in \IR$,
\item $\partial(\sigma^{it}\A\sigma^{-it})=V_t\partial(\A)$ for all $\A\in M_n(\IC)$, $t\in\IR$,
\item $\partial(\sigma^{1/2}\A^\ast \sigma^{-1/2})=\cJ\partial(\A)$ for all $\A\in M_n(\IC)$,
\item $F=\mathrm{lin}\{\partial(\A)\B\mid \A,\B\in M_n(\IC)\}$,
\item $\Gamma(\A,\B)=(\partial(\A)\vert\partial(\B))$ for all $\A,\B\in M_n(\IC)$.
\end{enumerate}
Moreover, if $(F^\prime,(V_t^\prime),\cJ^\prime,\partial^\prime)$ satisfies (a)--(h), then there exists a unique bimodule map $\Phi\colon F\to F^\prime$ such that $(\Phi(\xi)\vert\Phi(\eta))=(\xi\vert\eta)$ for all $\xi,\eta\in F$ and $\Phi\circ V_t=V_t^\prime\circ \Phi$, $\Phi\circ \cJ=\cJ^\prime\circ\Phi$, $\Phi\circ \partial=\partial^\prime$.
\end{proposition}
\begin{proof}
By \cite[Proposition 4.4]{Wir22a}, there exists a quadruple $(F^\prime,(V_t),\cJ,\partial)$ that satisfies all properties (a)--(h) except for (g) and with $F^\prime$ not necessarily finite-dimensional. Let $F=\mathrm{lin}\{\partial(A)B\mid A,B\in M_n(\IC)\}$. Then $F$ is finite-dimensional, the Leibniz rule ensures that $F$ is a sub-bimodule of $F^\prime$ and (a)--(f) ensure that $F$ is invariant under $(V_t)$ and $\cJ$. Thus $(F,(V_t|_F),\cJ|_F,\partial)$ satisfies (a)--(h).

The uniqueness property follows from \cite[Proposition 4.1]{Wir22a}.
\end{proof}

\begin{definition}
If $\L$ is the generator of a GNS-symmetric quantum Markov semigroup, we call a quadruple $(F,(V_t),\cJ,\partial)$ satisfying (a)--(h) from the previous proposition a \emph{first-order differential calculus} for $\L$.
\end{definition}

\begin{remark}
By the uniqueness part of \Cref{prop:fodc}, all notions defined in terms of a first-order differential calculus in the following are independent of the concrete choice of the quadruple $(F,(V_t),\cJ,\partial)$ and only depend on the  given quantum Markov semigroup.
\end{remark}

As a strongly continuous semigroup on a finite-dimensional Hilbert space, the map $t\mapsto V_t$ has a (unique) analytic extension $z\mapsto V_z$ to the complex plane. If $\rho\in M_n(\IC)$ is a positive invertible matrix, we define $V^\rho_z=\rho^{iz}\sigma^{-iz}V_z(\cdot)\sigma^{iz}\rho^{-iz}$ for $z\in\IC$. In particular, $V^\sigma_z=V_z$ for all $z\in\IC$.

\begin{lemma}
One has
\begin{equation*}
\tau\left[(\xi\vert V^\rho_{-i}\xi)\rho\right]=\tau\left[(V^\rho_{-i/2}\xi\vert V^\rho_{-i/2}\xi)\rho\right]
\end{equation*}
for all $\xi\in F$ and positive invertible $\rho\in M_n(\IC)$.
\end{lemma}
\begin{proof}
We have 
\begin{align*}
\tau(\xi\vert V^\rho_{-i}\xi)\rho)&=\tau\left[(\xi\vert \rho\sigma^{-1}(V_{-i}\xi)\sigma\rho^{-1})\rho\sigma^{-1}\sigma\right]\\
&=\tau\left[(\xi\vert V_{-i}(\sigma^{-1}\rho\xi))\sigma\right]\\
&=\tau\left[(V_{-i/2}\xi\vert \sigma^{-1/2}\rho\sigma^{-1/2}V_{-i/2}\xi)\sigma\right]\\
&=\tau\left[(\rho^{1/2}\sigma^{-1/2}V_{-i/2}\xi\vert \rho^{1/2}\sigma^{-1/2}V_{-i/2}\xi)\sigma^{1/2}\rho^{-1/2}\rho\rho^{-1/2}\sigma^{1/2}\right]\\
&=\tau\left[(\rho^{1/2}\sigma^{-1/2}(V_{-i/2}\xi)\sigma^{1/2}\rho^{-1/2}\vert \rho^{1/2}\sigma^{-1/2}(V_{-i/2}\xi)\sigma^{1/2}\rho^{-1/2})\rho\right]\\
&=\tau\left[(V^\rho_{-i/2}\xi\vert V^\rho_{-i/2}\xi)\rho\right].\qedhere
\end{align*}
\end{proof}

Let $\Lambda$ be an operator mean function and $f$ the associated operator monotone function given by $f(t)=\Lambda(1,t)$. For $\xi,\eta\in F$ we define 
$$\langle\xi,\eta\rangle_{\Lambda,\rho}=\tau\left[(\xi\vert f(V^\rho_{-i})\eta)\rho\right]\qquad \textnormal{and}\qquad
\norm{\xi}_{\Lambda,\rho}=\langle\xi,\xi\rangle_{\Lambda,\rho}^{1/2}.$$ 
By the previous lemma, the operator $V^\rho_{-i}$ is positive with respect to the inner product $\tau((\,\cdot\,\vert\,\cdot\,)\rho)$, hence this expression makes sense in terms of functional calculus for self-adjoint operators.

In the following lemma we denote by $L(\rho)$ and $R(\rho)$ the left and right multiplication operators on $M_n(\IC)$, that is, $L(\rho)(\A)=\rho \A$ and $R(\rho)(\A)=\A\rho$.

\begin{lemma}\label{lem:Alicki_fodc}
If $\L$ generates a quantum Markov semigroup that is GNS-symmetric with respect to $\sigma$ and $\V_j$, $\omega_j$ are as in Alicki's theorem, then $F=M_n(\IC)^d$ with inner product $((A_j)\vert (B_j))=\sum_j A_j^\ast B_j$, $(V_t \xi)_j=e^{i\omega_j t}\sigma^{it}\xi_j\sigma^{-it}$, $(\cJ \xi)_j=\sigma^{1/2}\xi_{j^\ast}^\ast\sigma^{-1/2}$ and $(\partial \A)_j=e^{-\omega_j/4}[\V_j,\A]$ is a first-order differential calculus for $\L$.

Moreover, if $\Lambda$ is an operator mean function, then
\begin{equation*}
\norm{\partial\A}_{\Lambda,\rho}^2=\sum_{j=1}^d \tau\left[[\V_j,\A]^\ast\Lambda(e^{\omega_j/2}L(\rho),e^{-\omega_j/2}R(\rho))[\V_j,\A]\right].
\end{equation*}
\end{lemma}
\begin{proof}
The first part follows by direct calculations (compare with \cite[Section 8.2]{Wir22b}).

Let $f$ be the operator monotone function associated with $\Lambda$. Note that 
$$(V^\rho_{-i}\xi)_j=\rho\sigma^{-1}(e^{\omega_j }\sigma\xi_j\sigma^{-1})\sigma\rho^{-1}
=e^{\omega_j }\rho \xi_j \rho^{-1}
=e^{\omega_j}L(\rho)R(\rho)^{-1}(\xi_j).$$
We have
\begin{align*}
\norm{\partial\A}_{\Lambda,\rho}^2&=\tau\left[(\partial\A\vert f(V^\rho_{-i})\partial\A)\rho\right]\\
&=\sum_{j=1}^d \tau\left[(\partial\A)_j^\ast f(e^{\omega_j}L(\rho)R(\rho)^{-1})(\partial\A)_j\rho\right]\\
&=\sum_{j=1}^d \tau\left[[\V_j,\A]^\ast f\left((e^{\omega_j/2}L(\rho))(e^{-\omega_j/2}R(\rho))^{-1}\right)\left(e^{-\omega_j/2}R(\rho)\right)[\V_j,\A]\right]\\
&=\sum_{j=1}^d \tau\left[[\V_j,\A]^\ast \Lambda(e^{\omega_j/2}L(\rho),e^{-\omega_j/2}R(\rho))[\V_j,\A]\right].\qedhere
\end{align*}
\end{proof}

\begin{definition}
Let $\Lambda$ be an operator mean function and $f(t)=\Lambda(1,t)$ for $t\geq 0$. We say that a GNS-symmetric quantum Markov semigroup with first-order differential calculus $(F,(V_t),\cJ,\partial)$ satisfies the gradient estimate $\mathrm{GE}_\Lambda(K,\infty)$ if 
\begin{equation*}
\norm{\partial(P_t(\A))}_{\Lambda,\rho}^2\leq e^{-2Kt}\norm{\partial(\A)}_{\Lambda,P_t^\dagger(\rho)}^2
\end{equation*}
for all self-adjoint $\A\in M_n(\IC)$, positive definite $\rho\in M_n(\IC)$ and $t\geq 0$.
\end{definition}

\begin{remark}
It follows easily from the uniqueness result for the first-order differential calculus associated with a GNS-symmetric quantum Markov semigroup that the gradient estimate does not depend on the choice of $(F,(V_t),\cJ,\partial)$. Furthermore, by the previous lemma, our gradient estimate is the same as the one used by Carlen and Maas to characterize their noncommutative Ricci curvature lower bounds \cite[Thereom 10.4]{CM20a} in the case when their differential structure is the one obtained from Alicki's theorem.
\end{remark}

\begin{remark}
Note that we restrict to self-adjoint matrices $A$ in the definition of the gradient estimate $\mathrm{GE}_\Lambda$. This is in line with the formulation from \cite{CM20a}. In particular, it is strong enough to imply the modified logarithmic Sobolev inequality. In contrast, in prior work by the second- and third-named author \cite{WZ21}, the gradient estimate was studied for all matrices $A\in M_n(\IC)$, not necessarily self-adjoint. As we show below in the case of dephasing semigroups (Subsection \ref{subsec:dephasing_QMS}), the constant in the gradient estimate $\mathrm{GE}_\Lambda$ may be better if one restricts to self-adjoint matrices $A$, and thus allows to obtain a stronger inequality for example in the modified logarithmic Sobolev inequality.
\end{remark}


\begin{definition}
Let $K\in\IR$. We say that a GNS-symmetric quantum Markov semigroup $(P_t)$ with first-order differential calculus $(F,(V_t),\cJ,\partial)$ has intertwining curvature lower bounded by $K$ if there exists a strongly continuous semigroup $(\vec P_t)$ on $F$ such that
\begin{itemize}
\item [(a)] $\vec P_t\partial=\partial P_t$ for all $t\in \IR$,
\item  [(b)] $\cJ\vec P_t=\vec P_t\cJ$ for all $t\in\IR$,
\item  [(c)] $\vec P_t V_s=V_s\vec P_t$ for all $s,t\in\IR$,
\item  [(d)] $(\vec P_t\xi\vert\vec P_t \xi)\leq e^{-2Kt}P_t(\xi\vert\xi)$ for all $t\in \IR$, $\xi\in F$.
\end{itemize}
\end{definition}

Again, intertwining curvature bounds can be rephrased in terms of a $\Gamma_2$-type criterion.

\begin{proposition}\label{prop:Gamma_2_intertwining_QMS}
A GNS-symmetric quantum Markov semigroup $(P_t)$ with first-order differential calculus $(F,\cJ,(V_t),\partial)$ has intertwining curvature bounded below by $K$ if and only if there exists a linear operator $\vec \L$ on $F$ such that $\vec\L\partial=\partial\vec\L$, $\vec \L\cJ=\cJ\vec \L$ and
\begin{equation*}
\frac 1 2\left((\vec\L\xi\vert\xi)+(\xi\vert\vec\L\xi)-\L(\xi\vert\xi)\right)\geq K(\xi\vert\xi)
\end{equation*}
for all $\xi\in F$.
\end{proposition}

\begin{theorem}\label{thm:grad_est_from_intertwining_QMS}
If a GNS-symmetric quantum Markov semigroup has intertwining curvature bounded below by $K$, then it satisfies the gradient estimate $\mathrm{GE}_\Lambda(K,\infty)$ for every operator mean function $\Lambda$.
\end{theorem}
\begin{proof}
Choose a strongly continuous semigroup $(\vec P_t)$ on $F$ as in the definition of intertwining curvature bounds. Let $H_1$ denote $F$ endowed with the inner product
\begin{equation*}
F\times F\to\IC,\,(\xi,\eta)\mapsto \tau((\xi\vert\eta)P_t^\dagger\rho)
\end{equation*}
and let $H_2$ denote $F$ endowed with the inner product
\begin{equation*}
F\times F\to\IC,\,(\xi,\eta)\mapsto \tau((\xi\vert\eta)\rho).
\end{equation*}
If we view $\vec P_t$ as an operator from $H_1$ to $H_2$, then the intertwining curvature bound implies $\vec P_t^\ast \vec P_t\leq e^{-2Kt}$.

Moreover, since $\vec P_t$ commutes with $\cJ$ and $(V_t)$ (and thus $V_{-i/2}$ as well), we have 
\begin{align*}
\tau\left[(\vec P_t\xi\vert V_{-i}^\rho\vec P_t\xi)\rho\right]&=\tau\left[(V^\rho_{-i/2}\vec P_t\xi\vert V^\rho_{-i/2}\vec P_t\xi)\rho\right]\\
&=\tau\left[(\rho^{1/2}\sigma^{-1/2} V_{-i/2}\vec P_t\xi\vert \rho^{1/2}\sigma^{-1/2} V_{-i/2}\vec P_t\xi)\sigma\right]\\
&=\tau\left[(\cJ(\rho^{1/2}\sigma^{-1/2} V_{-i/2}\vec P_t\xi)\vert\cJ(\rho^{1/2}\sigma^{-1/2} V_{-i/2}\vec P_t\xi))\sigma\right]\\
&=\tau\left[((\cJ V_{-i/2}\vec P_t\xi)\sigma^{1/2}\sigma^{-1/2}\rho^{1/2}\sigma^{-1/2}\vert (\cJ V_{-i/2}\vec P_t\xi)\sigma^{1/2}\sigma^{-1/2}\rho^{1/2}\sigma^{-1/2})\sigma\right]\\
&=\tau\left[((\vec P_t \cJ V_{-i/2}\xi)\rho^{1/2}\sigma^{-1/2}\vert (\vec P_t \cJ V_{-i/2}\xi)\rho^{1/2}\sigma^{-1/2})\sigma\right] \\
&=\tau\left[(\vec P_t\cJ V_{-i/2}\xi\vert \vec P_t\cJ V_{-i/2}\xi)\rho\right]\\
&\leq e^{-2Kt}\tau\left[(\cJ V_{-i/2}\xi\vert \cJ V_{-i/2}\xi)P_t^\dagger\rho\right].
\end{align*}
If we reverse the calculation up to the last equality, we see that the right side of the last inequality is $e^{-2Kt}\tau\left[(\xi\vert V^{P_t^\dagger\rho}_{-i}\xi)P_t^\dagger\rho\right]$. In other words, $\vec P_t^\ast V^\rho_{-i}\vec P_t\leq e^{-2Kt}V^{P_t^\dagger\rho}_{-i}$ if $V^\rho_{-i}$ is viewed as operator on $H_2$ and $V^{P_t^\dagger\rho}_{-i}$ is viewed as operator on $H_1$.

Now let $\Lambda$ be an operator mean function and $f$ the associated operator monotone function. Since every operator monotone function from $[0,\infty)$ to itself is operator concave, we deduce from Jensen's inequality for operators \cite[Theorem 2.5]{HP81} and the fact $e^{2Kt}P_t^\ast \vec P_t\leq \un$ that
\begin{equation*}
\vec P_t^\ast f(V^\rho_{-i})\vec P_t\leq e^{-2Kt}f(e^{2Kt}\vec P_t^\ast V^\rho_{-i}\vec P_t)\leq e^{-2Kt}f(V^\rho_{-i}),
\end{equation*}
which means
\begin{align*}
\norm{\vec P_t\xi}_{\Lambda,\rho}^2&=\tau\left[(\vec P_t \xi\vert f(V^\rho_{-i})\vec P_t\xi)\rho\right]\\
&=\langle \vec P_t\xi,f(V^\rho_{-i})\vec P_t\xi\rangle_{H_2}\\
&=\langle \xi,\vec P_t^\ast f(V^\rho_{-i})\vec P_t\xi\rangle_{H_1}\\
&\leq e^{-2Kt}\langle \xi,f(V^\rho_{-i})\xi\rangle_{H_1}\\
&= e^{-2Kt}\tau\left[(\xi\vert f(V^\rho_{-i})\xi)P_t^\dagger\rho\right]\\
&=e^{-2Kt}\norm{\xi}_{\Lambda,P_t^\dagger\rho}^2.
\end{align*}
Taking $\xi=\partial(\A)$ and observing that $\vec P_t\partial(\A)=\partial(P_t(\A))$, we obtain the gradient estimate $\mathrm{GE}_\Lambda(K,\infty)$.
\end{proof}

\begin{remark}\label{rmk:Bakry-Emery_from_intertwining_QMS}
Note that we did not use that $A$ is self-adjoint in the proof. In fact, for the Bakry--Émery estimate it is clear from the definition that an intertwining curvature lower bound $K$ implies $\Gamma(P_t(A))\leq e^{-2Kt}P_t\Gamma(A)$ for all $A\in M_n(\IC)$, not necessarily self-adjoint.
\end{remark}

\subsection{Intertwining curvature for dephasing semigroups}\label{subsec:dephasing_QMS}

Let $\sigma\in M_n(\IC)$ be a positive definite matrix with $\tau(\sigma)=1$ and let $E$ be a conditional expectation such that $\tau(E(A)\sigma)=\tau(A\sigma)$ for all $A\in M_n(\IC)$. The dephasing semigroup $(P_t)$ given by $P_t(\A)=e^{-t}\A+(1-e^{-t})E(\A)$ is a quantum Markov semigroup that is GNS-symmetric with respect to $\sigma$. Its generator is given by $\L(\A)=\A-E(\A)$.

We first give a Bakry--Émery curvature bound for $(P_t)$ in terms of the so-called \emph{Pimsner--Popa index} \cite{PP86}. The \emph{Pimsner--Popa index} $C(E)$ of $E$ is the supremum over all $C>0$ such that $C \A\leq E(\A)$ for all positive $\A\in M_n(\IC)$.

\begin{lemma}\label{lem:Bakry-Emery_dephasing}
Let $K=\frac 1 2+\frac{C(E)}{1+C(E)}$. The dephasing semigroup $(P_t)$ satisfies
\begin{equation*}
\Gamma(P_t \A)\leq e^{-2Kt}P_t\Gamma(\A)
\end{equation*}
for all $\A\in M_n(\IC)$.
\end{lemma}
\begin{proof}
As in the commutative case, it suffices to verify $\Gamma_2\geq K\Gamma$. We have
\begin{align*}
\Gamma(\A,\B)&=\frac 1 2\left[(\A-E(\A))^\ast \B+\A^\ast(\B-E(\B))-\A^\ast \B+E(\A^\ast \B)\right]\\
&=\frac 1 2\left[\A^\ast \B+E(\A^\ast \B)-\A^\ast E(\B)-E(\A)^\ast \B\right]\\
&=\frac 1 2(\A-E(\A))^\ast (\B-E(\B))+\frac 1 2(E(\A^\ast \B)-E(\A)^\ast E(\B)).
\end{align*}
In particular, $\Gamma(E(\A),\B)=0$ for all $\A,\B\in M_n(\IC)$.
Moreover,
\begin{align*}
\Gamma_2(\A)&=\frac 1 2\left[\Gamma(\A-E(\A),\A)+\Gamma(\A,\A-E(\A))-\Gamma(\A)+E(\Gamma(\A))\right]\\
&=\frac 1 2\left[\Gamma(\A)+E(\Gamma(\A))\right].
\end{align*}
Otherwise replacing $\A$ by $\A-E(\A)$, we can assume without loss of generality that $E(\A)=0$. Then
\begin{equation*}
	\Gamma (\A)=\frac 1 2 (\A^\ast \A+E(\A^\ast \A))\qquad \textnormal{and}\qquad E(\Gamma (\A))=E(\A^\ast \A).
\end{equation*}
Therefore, 
\begin{align*}
\Gamma_2(\A)-K\Gamma(\A)&=\frac 1 2 E(\A^\ast \A)-\frac {C(E)}{2(1+C(E))}(\A^\ast \A+E(\A^\ast \A))\\
&\geq \frac 1 2 E(\A^\ast \A)-\frac{C(E)}{2(1+C(E))}\left(1+\frac 1{C(E)}\right)E(\A^\ast \A)\\
&=0.\qedhere
\end{align*}
\end{proof}

\begin{proposition}\label{prop:intertwining_curv_dephasing}
The dephasing semigroup $(P_t)$ associated with the conditional expectation $E$ has intertwining curvature bounded below by $\frac 1 2+\frac{C(E)}{1+C(E)}$.
\end{proposition}
\begin{proof}
The proof is similar to the one for complete graphs in the previous section. Let $(F,(V_t),\cJ,\partial)$ be a first-order differential calculus for $(P_t)$. Equip $F$ with the inner product $\langle\xi,\eta\rangle=\tau((\xi\vert\eta)\sigma)$.

Note that $\ran \partial:=\{\partial A:A\in M_n(\IC)\}$ forms a closed subspace of $L^2(M_n(\com),\sigma)$. So $\xi\in M_n(\com)$ has the unique decomposition $\xi=\partial(\A)+\eta$ with some $\A\in M_n(\IC)$ and $\eta\perp\ran\partial$. Define $\vec \L\xi=\partial(\L(\A))+2K\eta$ with $K=\frac 1 2+\frac{C(E)}{1+C(E)}$. Note that $\vec \L$ is well-defined since $\partial \A=\partial \B$ implies 
$$\partial(\L(\A))=\partial \partial^{\dagger,\sigma}\partial A=\partial \partial^{\dagger,\sigma}\partial B=\partial(\L(\B)).
$$
We have 
\begin{align*}
(\xi\vert\xi)=\Gamma(\A)+(\partial(\A)\vert\eta)+(\eta\vert\partial(\A))+(\eta\vert \eta)
\end{align*}
and
\begin{align*}
(\vec \L\xi\vert \xi)+(\xi\vert\vec\L\xi)-\L(\xi\vert\xi)
&=\left(\partial \L \A\vert\partial \A\right)+\left(\partial \L \A\vert \eta\right)+2K \left(\eta \vert \partial A\right)+2K \left(\eta\vert\eta\right)\\
&\quad +\left(\partial\A\vert\partial \L A\right)+\left( \eta\vert \partial \L \A\right)+2K \left( \partial \A\vert \eta \right)+2K \left(\eta\vert\eta\right)\\
&\quad -\L\left(\partial \A\vert \partial \A\right)-\L\left(\partial \A\vert \eta\right)-\L\left(\eta\vert \partial A\right)-\L\left(\eta\vert \eta\right)\\
&=2\Gamma_2(\A)+\left(\partial \L \A\vert \eta\right)+\left( \eta\vert \partial \L \A\right)-\L\left(\partial \A\vert \eta\right)-\L\left(\eta\vert \partial \A\right)\\
&\quad +2K \left(\eta \vert \partial \A\right)+2K \left( \partial \A\vert \eta \right)+4K \left(\eta\vert\eta\right)-\L\left(\eta\vert \eta\right).
\end{align*}
Note that $\L(E(\A))=0$, hence $\partial(E(\A))=0$ and thus $\partial\L\A=\partial A$. So 
\begin{equation*}
	\left(\partial \L \A\vert \eta\right)+\left( \eta\vert \partial \L \A\right)-\L\left(\partial \A\vert \eta\right)-\L\left(\eta\vert \partial \A\right)
	=	E\left(\partial \A\vert \eta\right)+E\left(\eta\vert \partial \A\right).
\end{equation*}
Moreover, if $\B\in\operatorname{ran}E$, then $\partial B=0$ and thus $\partial (AB)=\partial (A)B$. So
\begin{align*}
\tau(E((\eta\vert\partial(\A)))\B\sigma)=\tau(E((\eta\vert \partial(\A)\B))\sigma)=\tau((\eta\vert \partial(\A \B))\sigma)=0.
\end{align*}
Taking $y=E((\eta\vert\partial\A))^\ast$ we obtain $E((\eta\vert\partial\A))=0$. Taking adjoints, we also get $E((\partial\A\vert\eta))=0$. Therefore,
\begin{align*}
(\vec \L\xi\vert \xi)+(\xi\vert\vec\L\xi)-\L(\xi\vert\xi)-2K(\xi\vert\xi)
&=2\Gamma_2(\A)+2K \left(\eta \vert \partial \A\right)+2K \left( \partial \A\vert \eta \right)-2K(\xi\vert\xi)\\
&\quad +4K \left(\eta\vert\eta\right)-\L\left(\eta\vert \eta\right)\\
&=2\Gamma_2(\A)-2K(\partial \A\vert\partial \A)-2K(\eta\vert\eta)+4K \left(\eta\vert\eta\right)-\L\left(\eta\vert \eta\right)\\
&= 2\Gamma_2(\A)-2K\Gamma(\A)+(2K-1)(\eta\vert\eta)+E(\eta\vert\eta).
\end{align*}
The term $2\Gamma_2(A)-2K\Gamma(A)$ is non-negative by the previous lemma. The term  $(2K-1)(\eta\vert\eta)+E(\eta\vert\eta)$ is non-negative since $2K-1\ge 0$. Hence $(P_t)$ has intertwining curvature bounded below by $K$.
\end{proof}

\begin{example}
If $E(\A)=\tau(\A)\un$, then $(P_t)$ is called the \emph{depolarizing semigroup} on $M_n(\IC)$. In this case, $C(E)=\frac 1 n$, and we obtain that $(P_t)$ has intertwining curvature bounded below by $\frac 1 2+\frac 1 {n+1}$. We will show in the next subsection that $(P_t)$ satisfies the gradient estimate $\mathrm{GE}_\Lambda(\frac 1 2+\frac 1 n,\infty)$ for all operator mean functions $\Lambda$, which is stronger than bound we obtain from the intertwining criterion in the light of \Cref{thm:grad_est_from_intertwining_QMS}. 

In fact, the intertwining curvature bound $\frac 1 2+\frac 1{n+1}$ is sharp for $(P_t)$: As discussed in \Cref{rmk:Bakry-Emery_from_intertwining_QMS}, intertwining curvature bounded below by $K$ implies that
\begin{equation*}
\Gamma_2(\A)\geq K\Gamma(\A)
\end{equation*}
for all $\A\in M_n(\IC)$. As we we have shown in the proof of \Cref{lem:Bakry-Emery_dephasing},
\begin{align*}
\Gamma(\A)&=\frac 1 2(\A^\ast \A+\tau(\A^\ast \A)-\overline{\tau(\A)}\A-\tau(\A)\A^\ast)\\
\Gamma_2(\A)&=\frac 1 2(\Gamma(\A)+\tau(\A^\ast \A)-\abs{\tau(\A)}^2).
\end{align*}
Hence
\begin{align*}
	2\Gamma_2(\A)-2K\Gamma(\A)&=(1-2K)\Gamma(\A)+\tau(\A^\ast \A)-\abs{\tau(\A)}^2.
\end{align*}
For $\A=E_{jk}$ with $j\neq k$, it follows that
\begin{align*}
	2\Gamma_2(\A)-2K\Gamma(\A)&=\frac 1 2(1-2K)\left(E_{kk}+\frac 1 n\right)+\frac 1 n,
\end{align*}
which is positive if and only if $K\leq \frac 1 2 +\frac 1 {n+1}$.
\end{example}

\subsection{Intertwining curvature for generators with commuting jump operators}
In this subsection, we show a permanence property of intertwining curvature for quantum Markov semigroups that commute in a very strong sense.

\begin{theorem}\label{thm:intertwining_QMS_very_commuting}
For $m\in\mathbb N$ let $d_1,\dots,d_m\in\mathbb N$, $v_{j,k}\in M_n(\IC)$ and $\omega_{j,k}\in\IR$ for $1\leq k\leq m$, $1\leq j\leq d_k$ such that $\{v_{j,k}\mid 1\leq j\leq d_k\}=\{v_{j,k}^\ast\mid 1\leq j\leq d_k\}$ for $1\leq k\leq m$ and $\sigma v_{j,k}\sigma^{-1}=e^{-\omega_{j,k}}v_{j,k}$ for $1\leq k\leq m$, $1\leq j\leq d_k$. Assume that $[v_{i,k},v_{j,l}]=0$ for $1\leq k,l\leq m$, $1\leq i\leq d_k$, $1\leq j\leq d_l$ with $k\neq l$.

For $1\leq k\leq m$ let
\begin{equation*}
\L_k\colon M_n(\IC)\to M_n(\IC),\,A\mapsto \sum_{j=1}^{d_k}\big(e^{-\omega_{j,k}/2} v_{j,k}^\ast[v_{j,k},A]-e^{\omega_{j,k}/2}[v_{j,k},A]v_{j,k}^\ast\big)
\end{equation*}
and let $\L=\sum_{k=1}^m \L_k$.

Then $\L$ generates a quantum Markov semigroup $(P_t)$ that is GNS-symmetric with respect to $\sigma$, and if $(e^{-t\L_k})$ (which is also a quantum Markov semigroup that is GNS-symmetric with respect to $\sigma$) has intertwining curvature bounded below by $K$ for $1\leq k\leq m$, then so does $(P_t)$.
\end{theorem}
\begin{proof}
Let $P_t^{(k)}=e^{-t\L_k}$. By Alicki's theorem, $(P_t^{(k)})$ is a quantum Markov semigroup that is GNS-symmetric with respect to $\sigma$. Let $\partial_{j,k}(A)=e^{-\omega_{j,k}/4}[v_{j,k},A]$ and recall that $\partial_{j,k}^{\dagger,\sigma}$ is the adjoint with respect to the GNS inner product $\langle A,B\rangle_\sigma=\tau(A^\ast B\sigma)$.

We have
\begin{align*}
\langle A,\partial_{j,k}^{\dagger,\sigma}(B)\rangle_\sigma&=\langle \partial_{j,k}(A),B\rangle_\sigma\\
&=e^{-\omega_j/4}\tau([v_{j,k},A]^\ast B\sigma)\\
&=e^{-\omega_{j,k}/4}\tau\left[A^\ast(v_{j,k}^\ast B-B\sigma v_{j,k}^\ast\sigma^{-1})\sigma\right]\\
&=\tau\left[A^\ast(e^{-\omega_{j,k}/4}v_{j,k}^\ast B-e^{3\omega_{j,k}/4}Bv_{j,k}^\ast)\sigma\right].
\end{align*}
Thus $\partial_{j,k}^{\dagger,\sigma}(A)=e^{-\omega_{j,k}/4}v_{j,k}^\ast A-e^{3\omega_{j,k}/4}A v_{j,k}^\ast$ and
\begin{align*}
\sum_{j=1}^{d_k}\partial_{j,k}^{\dagger,\sigma}\partial_{j,k}(A)=\sum_{j=1}^{d_k}(e^{-\omega_{j,k}/2}v_{j,k}^\ast[v_{j,k},A]-e^{\omega_{j,k}/2}[v_{j,k},A]v_{j,k}^\ast)=\L_k(A).
\end{align*}
The commutation relation $[v_{i,k},v_{j,l}]=0$ for $k\neq l$ implies $[\partial_{i,k},\partial_{j,l}]=0$ and
\begin{align*}
\partial_{i,k}\partial_{j,l}^{\dagger,\sigma}(A)&=e^{-\omega_{i,k}/4-\omega_{j,l}/4}[v_{i,k},v_{j,l}^\ast A]-e^{-\omega_{i,k}/4+3\omega_{j,l}/4}[v_{i,k},Av_{j,l}^\ast]\\
&=e^{-\omega_{i,k}/4-\omega_{j,l}/4}(\underbrace{[v_{i,k},v_{j,l}^\ast]}_{=0}A+v_{j,l}^\ast[v_{i,k},A])-e^{-\omega_{i,k}/4+3\omega_{j,l}/4}([v_{i,k},A] v_{j,l}^\ast+A\underbrace{[v_{i,k},v_{j,l}^\ast]}_{=0})\\
&=\partial_{j,l}^{\dagger,\sigma}\partial_{i,k}(A).
\end{align*}
Here we used the fact that since $\{v_{j,l}\mid 1\leq j\leq d_l\}=\{v_{j,l}^\ast\mid 1\leq j\leq d_l\}$, we also have $[v_{i,k},v_{j,l}^\ast]=0$.

Hence $[\L_k,\L_l]=0$ for $1\leq k,l\leq m$ and $P_t=P_t^{(1)}\dots P_t^{(m)}$. From this we deduce that $(P_t)$ is a quantum Markov semigroup that is GNS-symmetric with respect to $\sigma$.

Let $(F_k,(V_t^{(k)}),\cJ_k,\partial_k)$ be the first-order differential calculus for $\L_k$ as in \Cref{lem:Alicki_fodc}. In particular, $F_k\subset M_n(\IC)^{d_k}$ and $\partial_k(A)=(\partial_{j,k}(A))_j$. Let $\partial(A)=(\partial_k(A))_k\in \bigoplus_k F_k$, $F=\mathrm{lin}\{\partial(A)B\mid A,B\in M_n(\IC)\}\subset \bigoplus_k F_k$, $V_t=\bigoplus_k V^{(k)}_t$, $\cJ=\bigoplus_k \cJ_k$. Direct computation shows that $(F,(V_t),\cJ,\partial)$ is a first-order differential calculus for $\L$.

Let $\vec\L_k$ be an operator that realizes the intertwining curvature lower bound $K$ for $\L_k$ and let $\vec\L=\bigoplus_k (\vec \L_k+\sum_{l\neq k}\bigoplus_{j=1}^{d_k}\L_l)$. We have
\begin{align*}
\vec\L(\partial A)&=(\vec\L_k(\partial_k A))_k+\sum_{l\neq k}((\L_l \partial_{j,k}A)_j)_k\\
&=(\partial_k(\L_k A))_k+\sum_{l\neq k}(\partial_k(\L_l A))_k\\
&=\partial(\L_k(A)+\sum_{l\neq k}\L_l(A))\\
&=\partial(\L(A)).
\end{align*}
The relations $\cJ\vec\L=\vec\L\cJ$ and $V_t\vec\L=\vec\L V_t$ follow directly from the corresponding identities for $\vec\L_k$.

For $\xi=(\xi_k)_k\in F$  we have
\begin{align*}
&\quad(\vec \L\xi\vert\xi)+(\xi\vert\vec \L\xi)-\L(\xi\vert\xi)\\
&=\sum_{k=1}^m \left((\vec\L_k\xi_k\vert\xi_k)+\sum_{l\neq k}(\L_l^{\oplus d_k}\xi_k\vert\xi_k)+(\xi_k\vert\vec \L_k\xi_k)+\sum_{l\neq k}(\xi_k\vert \L_l^{\oplus d_k}\xi_k)-\sum_{l=1}^m\L_l(\xi_k\vert\xi_k)\right)\\
&=\sum_{k=1}^m \left((\vec\L_k\xi_k\vert\xi_k)+(\xi_k\vert\vec\L_k\xi_k)-\L_k(\xi_k\vert\xi_k)+\sum_{l\neq k}(\L_l^{\oplus d_k}\xi_k\vert\xi_k)+(\xi_k\vert\L_l^{\oplus d_k}\xi_k)-\L_l(\xi_k\vert\xi_k))\right).
\end{align*}
By the intertwining curvature bound for $\L_k$, we have $(\vec\L_k\xi_k\vert\xi_k)+(\xi_k\vert\vec\L_k\xi_k)-\L_k(\xi_k\vert\xi_k)\geq 2K(\xi_k\vert\xi_k)$.

Moreover, writing $\xi_k=(A_{j,k})_j\in M_n(\IC)^{\oplus d_k}$, one has
\begin{align*}
&(\L_l^{\oplus d_k}\xi_k\vert\xi_k)+(\xi_k\vert\L_l^{\oplus d_k}\xi_k)-\L_l(\xi_k\vert\xi_k)\\
&\quad=\sum_{j=1}^{d_k}\L_l(A_{j,k})^\ast A_{j,k}+A_{j,k}^\ast \L_l(A_{j,k})-\L_l(A_{j,k}^\ast A_{j,k})\\
&\quad\geq 0.
\end{align*}
Therefore,
\begin{align*}
(\vec \L\xi\vert\xi)+(\xi\vert\vec \L\xi)-\L(\xi\vert\xi)\geq 2K\sum_{k=1}^m (\xi_k\vert\xi_k)=2K(\xi\vert\xi).
\end{align*}
Now it follows from \Cref{prop:Gamma_2_intertwining_QMS} that $\L$ has intertwining curvature bounded below by $K$.
\end{proof}

\begin{corollary}
If $p_1,\dots,p_d\in M_n(\IC)$ are commuting projections and $\alpha_1,\dots,\alpha_d>0$, then 
\begin{equation*}
\L\colon M_n(\IC)\to M_n(\IC),\,A\mapsto\sum_{k=1}^d \alpha_k(p_k A+Ap_k-2p_k A p_k)
\end{equation*}
generates a tracially symmetric quantum Markov semigroup on $M_n(\IC)$ that has intertwining curvature bounded below by $\min_{1\leq k\leq d}\alpha_k$.
\end{corollary}
\begin{proof}
Let $\L_k(A)=\alpha_k(p_k A+A p_k-2p_k A p_k)$. By the previous theorem it suffices to show that $\L_k$ has intertwining curvature bounded below by $\alpha_k$.

Every $A\in M_n(\IC)$ can be written as $A=[p_k,B]+C$ with $B,C\in M_n(\IC)$ such that $[p_k,C]=0$. For example, take $B=A-p_kAp_k$ and $C=p_k A p_k+(\un-p_k)A(\un-p_k).$ Let $\vec\L_k(A)=[p_k,\L_k(B)]+2\alpha_k C$. We have
\begin{align*}
&\quad\;\vec \L_k(A)^\ast A+A^\ast\vec\L_k(A)-\L_k(A^\ast A)-2\alpha_k A^\ast A\\
&=[p_k,\L_k(B)]^\ast [p_k,B]+[p_k,B]^\ast [p_k,\L_k(B)]-\L_k([p_k,B]^\ast[p_k,B])-2\alpha_k[p_k,B]^\ast[p_k,B]\\
&\quad+[p_k,\L_k(B)]^\ast C+2\alpha_k[p_k,B]^\ast C-\L_k([p_k,B]^\ast C)-2\alpha_k[p_k,B]^\ast C\\
&\quad+C^\ast[p_k,\L_k(B)]+2\alpha_kC^\ast[p_k,B]-\L_k(C^\ast [p_k,B])-2\alpha_k C^\ast [p_k,B]\\
&\quad+2\alpha_k C^\ast C-\L_k(C^\ast C).
\end{align*}
Since $[p_k,C]=0$ and $\L_k=[p_k,[p_k,\cdot\,]]$, we have $\L_k([p_k,B]^\ast C)=\L_k([p_k,B])^\ast C=[p_k,\L_k(B)]^\ast C$. Hence the second line vanishes. The third line is the adjoint of the second line and thus also zero. Moreover, $[p_k,C]=0$ implies $[p_k,C^\ast C]=0$ and hence $\L_k(C^\ast C)=0$. Therefore, the fourth line is positive.

It remains to treat the first line, which is nothing but the Bakry--Émery criterion for $\L_k$. With $\L_k=\alpha_k[p_k,[p_k,\cdot\,]]$ we get
\begin{align*}
[p_k,\L_k(B)]&=\alpha_k[p_k,p_kB+Bp_k-2p_k B p_k]\\
&=\alpha_k(p_k B-p_k B p_k+p_k B p_k-B p_k)\\
&=\alpha_k[p_k, B]
\end{align*}
and
\begin{align*}
\L_k([p_k,B]^\ast [p_k,B])&=[p_k,[p_k,B^\ast p_k B+p_k B^\ast Bp_k-p_k B^\ast p_k B-B^\ast p_k B p_k]]=0.
\end{align*}
Thus
\begin{align*}
&[p_k,\L_k(B)]^\ast [p_k,B]+[p_k,B]^\ast [p_k,\L_k(B)]-\L_k([p_k,B]^\ast[p_k,B])-2\alpha_k[p_k,B]^\ast[p_k,B]\\
&\quad=2\alpha_k[p_k,B]^\ast[p_k,B]-2\alpha_k[p_k,B]^\ast[p_k,B]\\
&\quad=0.
\end{align*}
Altogether we have shown that 
\begin{equation*}
\vec \L_k(A)^\ast A+A^\ast\vec\L_k(A)-\L_k(A^\ast A)=2\alpha_k C^\ast C+2\alpha_k A^\ast A\geq 2\alpha_k A^\ast A.
\end{equation*}
It follows from \Cref{prop:Gamma_2_intertwining_QMS} that $\L_k$ has intertwining curvature bounded below by $\alpha_k$.
\end{proof}

\begin{remark}
This result shows that the intertwining curvature bound obtained in \Cref{prop:intertwining_curv_dephasing} is in general not optimal. Indeed, let $\L(A)=pA+Ap-2pAp$. As we have just seen, this operator has intertwining curvature bounded below by $1$. However,
\begin{equation*}
pA+Ap-2pAp=A-(pAp+(1-p)A(1-p)),
\end{equation*}
and $E(A)=pAp+(\un-p)A(\un-p)$ defines a conditional expectation. This conditional expectation is not equal to the identity unless $p=0$ or $p=\un$, hence $C(E)<1$. Hence the intertwining curvature lower bound $\frac 1 2+\frac{C(E)}{1+C(E)}$ from \Cref{prop:intertwining_curv_dephasing} is strictly smaller than $1$.
\end{remark}

\begin{remark}
The gradient estimate $\mathrm{GE}_\Lambda(1,\infty)$ in this setting (and more generally when $M_n(\IC)$ is replaced by a finite von Neumann algebra) was obtained in \cite[Theorem 5.1]{WZ21} by a non-intertwining approach.
\end{remark}

\subsection{Gradient estimate for the depolarizing semigroup}
In this section we show that for depolarizing semigroups, the gradient estimate obtained from the intertwining curvature lower bound in \Cref{prop:intertwining_curv_dephasing} can be improved and in particular, the sharp modified logarithmic Sobolev inequality for the depolarizing semigroup on qubits can be derived from the entropic curvature lower bound.

Let $(P_t)$ be the depolarizing semigroup on $M_n(\IC)$, that is, $P_t(x)=e^{-t}\A+(1-e^{-t})\tau(\A)\un$. This semigroup is tracially symmetric. Its generator is given by $\L(\A)=\A-\tau(\A)\un$. We first give a convenient presentation of the first-order differential calculus associated with $(P_t)$.

\begin{lemma}\label{lem:diff_calc_depolarizing}
Let $F=\{\sum_j \A_j\otimes \B_j\in M_n(\IC)\otimes M_n(\IC):\sum_j \A_j \B_j=0\}$ with $\A(a\otimes b)B=\A a\otimes b\B$ and $(\A_1\otimes \B_1\vert \A_2\otimes \B_2)=\frac 1 2\tau(\A_1^\ast \A_2)\B_1^\ast \B_2$, $V_t=\mathrm{id}$, $\cJ(\A\otimes \B)=-\B^\ast\otimes \A^\ast$, $\partial(\A)=\A\otimes \un-\un\otimes \A$. Then $(F,(V_t),\cJ,\partial)$ is a first-order differential calculus.
\end{lemma}
\begin{proof}
Properties (a)--(f) are easy to see. For property (g), note that if $\sum_j \A_j \B_j=0$, then
\begin{equation*}
\sum_j \A_j\otimes \B_j=\sum_j \A_j\otimes \B_j-1\otimes \A_j \B_j=\sum_j \partial(\A_j)\B_j.
\end{equation*}
Finally,
\begin{align*}
\Gamma(\A,\B)&=\frac 1 2(\A^\ast \B+\tau(\A^\ast \B)-\overline{\tau(\A)}\B-\A^\ast\tau(\B))\\
&=(\un\otimes \A\vert \un\otimes \B)+(\A\otimes \un\vert \B\otimes \un)-(\A\otimes \un\vert \un\otimes \B)-(\un\otimes \A\vert \B\otimes \un)\\
&=(\A\otimes \un-\un\otimes \B\vert \B\otimes \un-\un\otimes \B)\\
&=(\partial(\A)\vert\partial(\B)).\qedhere
\end{align*}
\end{proof}

\begin{remark}
Since $V_t=\mathrm{id}$, we have $V^\rho_{z}\xi=\rho^{iz}\xi\rho^{-iz}$. Writing $L(\rho)\xi=\rho\xi$, $R(\rho)\xi=\xi\rho$, we conclude
\begin{align*}
\langle\xi,\eta\rangle_{\Lambda,\rho}=\tau((\xi\vert f(V^\rho_{-i})\eta)\rho)=\tau((\xi\vert f(L(\rho)R(\rho)^{-1})R(\rho)\eta))=\frac 1 2\langle \xi,\Lambda(L(\rho),R(\rho))\eta\rangle_{\tau\otimes\tau}.
\end{align*}
\end{remark}

\begin{lemma}\label{lem:orthogonality_depolarizing}
Let $(F,(V_t),\cJ,\partial)$ be a first-order differential calculus for $(P_t)$. Let $(e_i)$ be an orthonormal basis of $\IC^n$ and $E_{ij}=\langle e_j,\cdot\,\rangle e_i$. If $\rho=\sum_{p=1}^n \lambda_p E_{pp}$ with $\lambda_p>0$, and $1\leq i,j,k,l\leq n$ with $k\neq l$, then 
\begin{equation*}
\langle \partial(E_{ij}),\partial(E_{kl})\rangle_{\Lambda,\rho}=0
\end{equation*}
unless $i=k$ and $j=l$.

In particular, $\langle \partial(A),\partial(B)\rangle_{\Lambda,\rho}=0$ if $A$ is diagonal and $B$ has vanishing diagonal in the basis $(e_i)$.
\end{lemma}
\begin{proof}
For any $\A,\B\in M_n(\IC)$ we have
\begin{equation*}
\langle\partial(\A),\partial(\B)\rangle_{\Lambda,\rho}=\frac 1 2 \sum_{p,q=1}^n\Lambda(\lambda_p,\lambda_q)\langle \partial(\A),E_{pp}\partial(\B)E_{qq}\rangle_{\tau\otimes\tau}
\end{equation*}
and
\begin{align*}\label{eq:*}
&\langle \partial(\A),E_{pp}\partial(\B)E_{qq}\rangle_{\tau\otimes\tau} \\ 
&\quad=\tau(\A^\ast E_{pp}\B)\tau(E_{qq})-\tau(\A^\ast E_{pp})\tau(\B E_{qq})-\tau(E_{pp}\B)\tau(\A^\ast E_{qq})+\tau(E_{pp})\tau(\A^\ast \B E_{qq}).\tag{$*$}
\end{align*}
If $\A=E_{ij}$ and $B=E_{kl}$, then 
\begin{align*}
&\langle \partial(\A),E_{pp}\partial(\B)E_{qq}\rangle_{\tau\otimes\tau} \\ 
&\quad=\tau(E_{ji} E_{pp} E_{kl})\tau(E_{qq})-\tau(E_{ji} E_{pp})\tau(E_{kl} E_{qq})-\tau(E_{pp}E_{kl})\tau(E_{ji} E_{qq})+\tau(E_{pp})\tau(E_{ji} E_{kl} E_{qq})\\
&\quad=\frac{1}{n^2}[\delta_{ip}\delta_{pk}\delta_{lj}-2\delta_{ip}\delta_{jp}\delta_{kq}\delta_{lq}+\delta_{ik}\delta_{lq}\delta_{qj}].
\end{align*}
The first and the last term vanish unless $(i,j)=(k,l)$, while the second term vanishes unless $i=j,k=l$, which is ruled out by our assumption. In particular, $\langle\partial(E_{ii}),\partial(E_{kl})\rangle_{\Lambda,\rho}=0$ if $k\neq l$, which implies the last part.
\end{proof}

\begin{theorem}\label{thm:GE_depolarizing}
The depolarizing semigroup on $M_n(\IC)$ satisfies the gradient estimate $\mathrm{GE}_{\Lambda}(\frac 12+\frac 1 n,\infty)$ for every operator mean function $\Lambda$.
\end{theorem}

\begin{proof}
We have to prove that
\begin{equation}\label{ineq:GE depolarizing}
\norm{\partial P_t(\A)}_{\Lambda,\rho}^2\leq e^{-(1+\frac 2 n)t}\norm{\partial \A}_{\Lambda,P_t(\rho)}^2
\end{equation}
for every self-adjoint $\A\in M_n(\IC)$, positive definite $\rho\in M_n(\IC)$ and $t\geq 0$.

Since $\partial P_t(\A)=e^{-t}\partial \A$, this reduces to
\begin{equation*}
\norm{\partial \A}_{\Lambda,\rho}^2\leq e^{(1-\frac 2 n)t}\norm{\partial \A}_{\Lambda,P_t(\rho)}^2.
\end{equation*}
Hence it is sufficient to show that
\begin{equation}\label{ineq:monotonicity t=0}
t\mapsto e^{(1-\frac 2 n)t}\norm{\partial\A}_{\Lambda,P_t(\rho)}^2
\end{equation}
has a minimum at $t=0$.

Let $\lambda_p$ be the eigenvalue of $\rho$ to the unit eigenvector $e_p$ and let $\lambda_p(t)=e^{-t}\lambda_p+(1-e^{-t})\tau(\rho)$. Then $\lambda_p(t)$ is the eigenvalue of $P_t(\rho)$ to the eigenvector $e_p$. By \Cref{lem:orthogonality_depolarizing}, it suffices to prove the assertion for $\A$ diagonal in the basis $(e_p)$ and $\A$ having vanishing diagonal in the basis $(e_p)$, as the cross terms vanish following the computations in the proof of the previous lemma. The result for diagonal $\A$ follows from our result on complete graphs in \Cref{ex:complete_graph}.

So it reduces to showing \eqref{ineq:monotonicity t=0} for self-adjoint $A$ with vanishing diagonal. Recall that  for $k\neq l$, the proof of Lemma \ref{lem:orthogonality_depolarizing} gives $\langle \partial(E_{kl}),E_{pp}\partial(E_{kl}) E_{qq}\rangle_{\tau\otimes\tau}=0$ unless $(i,j)=(k,l)$. Since $A$ is self-adjoint, 
$$\langle \partial(A),E_{pp}\partial(A) E_{qq}\rangle_{\tau\otimes\tau}=\langle \partial(A),E_{qq}\partial(A) E_{pp}\rangle_{\tau\otimes\tau}.$$
Writing $A=\sum_{k\neq l}A_{kl}E_{kl}$, we have
\begin{align*}
&\quad e^{(1-\frac 2 n)t}\norm{\partial(A)}_{\Lambda,P_t(\rho)}^2\\
&=\frac 1 2 \langle\partial(A),\Lambda(P_t(\rho)\otimes\un,\un\otimes P_t(\rho))\partial(A)\rangle_{\tau\otimes\tau}\\
&=\frac 1 2 e^{(1-\frac 2 n)t}\sum_{p,q=1}^n \Lambda(\lambda_p(t),\lambda_q(t))\langle \partial(A),E_{pp}\partial(A) E_{qq}\rangle_{\tau\otimes\tau}\\
&=\frac 1 2 e^{(1-\frac 2 n)t}\sum_{p,q=1}^n \Lambda(\lambda_q(t),\lambda_p(t))\langle \partial(A),E_{pp}\partial(A) E_{qq}\rangle_{\tau\otimes\tau}\\
&=\frac 1 2 \sum_{k\neq l}\sum_{p\neq q}^n e^{(1-\frac 2 n)t}\frac{\Lambda(\lambda_p(t),\lambda_q(t))+\Lambda(\lambda_q(t),\lambda_p(t))}{2}|A_{kl}|^2\langle \partial(E_{kl}),E_{pp}\partial(E_{kl}) E_{qq}\rangle_{\tau\otimes\tau}\\
\end{align*}
and to prove \eqref{ineq:monotonicity t=0} we divide the above into two parts $p\neq q$ and $p=q$:
\begin{align*}
&\quad e^{(1-\frac 2 n)t}\norm{\partial(A)}_{\Lambda,P_t(\rho)}^2\\
&=\frac 1 2 \sum_{k\neq l}\sum_{p\neq q} e^{(1-\frac 2 n)t}\frac{\Lambda(\lambda_p(t),\lambda_q(t))+\Lambda(\lambda_q(t),\lambda_p(t))}{2}|A_{kl}|^2\langle \partial(E_{kl}),E_{pp}\partial(E_{kl}) E_{qq}\rangle_{\tau\otimes\tau}\\
&\quad +\frac 1 2 \sum_{k\neq l}\sum_{p=1}^n e^{(1-\frac 2 n)t}\lambda_p(t)|A_{kl}|^2\langle \partial(E_{kl}),E_{pp}\partial(E_{kl}) E_{pp}\rangle_{\tau\otimes\tau}.
\end{align*}
Thus it suffices to show that 
\begin{equation}\label{ineq:p neq q}
e^{(1-\frac 2 n)t}\frac{\Lambda(\lambda_p(t),\lambda_q(t))+\Lambda(\lambda_q(t),\lambda_p(t))}{2}
\ge \frac{\Lambda(\lambda_p,\lambda_q)+\Lambda(\lambda_q,\lambda_p)}{2},\qquad p\neq q
\end{equation}
and for $k\neq l$
\begin{equation}\label{ineq: p=q}
\sum_{p=1}^n e^{(1-\frac 2 n)t}\lambda_p(t)\langle \partial(E_{kl}),E_{pp}\partial(E_{kl}) E_{pp}\rangle_{\tau\otimes\tau}
\ge \sum_{p=1}^n \lambda_p\langle \partial(E_{kl}),E_{pp}\partial(E_{kl}) E_{pp}\rangle_{\tau\otimes\tau}.
\end{equation}

In the former case $p\neq q$, we can use concavity of $\Lambda$ to get
\begin{equation*}
e^{(1-\frac 2 n)t}\frac{\Lambda(\lambda_p(t),\lambda_q(t))+\Lambda(\lambda_q(t),\lambda_p(t))}{2}
\geq e^{-\frac 2 n t}\frac{\Lambda(\lambda_p,\lambda_q)+\Lambda(\lambda_q,\lambda_p)}{2}+(e^{(1-\frac 2 n)t}-e^{-\frac 2 n t})\tau(\rho).
\end{equation*}
As $p\neq q$, we have $\frac 1 2(\lambda_p+\lambda_q)\leq \frac n 2\tau(\rho)$. Moreover, $(a,b)\mapsto \frac{1}{2}[\Lambda(a,b)+\Lambda(b,a)]$ is a symmetric mean so that 
$$\frac{1}{2}[\Lambda(a,b)+\Lambda(b,a)]\le \frac{a+b}{2}.$$
Thus
\begin{equation*}
(e^{(1-\frac 2 n)t}-e^{-\frac 2 n t})\tau(\rho)
\ge \frac{1}{n}(e^{(1-\frac 2 n)t}-e^{-\frac 2 n t})(\lambda_p+\lambda_q)
\ge \frac{2}{n}(e^{(1-\frac 2 n)t}-e^{-\frac 2 n t})\frac{\Lambda(\lambda_p,\lambda_q)+\Lambda(\lambda_q,\lambda_p)}{2}
\end{equation*}
To show \eqref{ineq:p neq q}, it remains to prove
$$
e^{-\frac{2}{n}t}+\frac{2}{n}(e^{(1-\frac 2 n)t}-e^{-\frac 2 n t})\ge 1.
$$
This holds for all $t\ge 0$ since both sides agree at $t=0$, and the derivative of the left side is $(1-\frac{2}{n})\frac{2}{n}e^{-\frac{2}{n}t}(e^t-1)\ge 0$. Here we used the fact that $n\ge 2$.


In the latter case $p=q$, we have by the proof of Lemma \ref{lem:orthogonality_depolarizing} that 
\begin{equation*}
\sum_{p=1}^n e^{(1-\frac 2 n)t}\lambda_p(t)\langle \partial(E_{kl}),E_{pp}\partial(E_{kl}) E_{pp}\rangle_{\tau\otimes\tau}
=\frac{1}{n^2}\sum_{p=1}^n e^{(1-\frac 2 n)t}\lambda_p(t)(\delta_{pk}+\delta_{pl})
=\frac{1}{n^2} e^{(1-\frac 2 n)t}(\lambda_k(t)+\lambda_l(t)).
\end{equation*}
So \eqref{ineq: p=q} is nothing but 
\begin{equation*}
\frac{1}{n^2} e^{(1-\frac 2 n)t}(\lambda_k(t)+\lambda_l(t))
\ge \frac{1}{n^2} (\lambda_k+\lambda_l),\qquad k\neq l
\end{equation*}
which follows from \eqref{ineq:p neq q} by choosing $\Lambda$ there to be the left trivial mean.

Altogether, we finish the proof of the theorem. 
\end{proof}

\begin{remark}
In the above theorem, \eqref{ineq:GE depolarizing} also holds for all symmetric operator mean $\Lambda$ and all $A\in M_n(\com)$ that is not necessarily self-adjoint. In fact, in this case we don't need the symmetrization of the operator mean using the self-adjointness of $A$, and we may divide 
\begin{align*}
e^{(1-\frac 2 n)t}\norm{\partial(A)}_{\Lambda,P_t(\rho)}^2=\frac 1 2 e^{(1-\frac 2 n)t}\sum_{p,q=1}^n \Lambda(\lambda_p(t),\lambda_q(t))\langle \partial(A),E_{pp}\partial(A) E_{qq}\rangle_{\tau\otimes\tau}
\end{align*}
directly into $p\neq q$ and $p=q$ cases and the rest of the proof is almost identical. In view of Example \ref{ex:complete_graph}, we cannot remove the symmetry of $\Lambda$ and the self-adjointness of $A$ simultaneously.
\end{remark}
\begin{remark}
In the case when $\Lambda$ is the logarithmic mean, this result improves upon previous work by Carlen--Maas \cite{CM20b}, who showed $\mathrm{GE}_\Lambda\left(\frac 1 2+\frac 1{2n},\infty\right)$.
\end{remark}

\begin{remark}
If $\Lambda$ is the logarithmic mean, the gradient estimate $\mathrm{GE}_\Lambda\left(\frac 1 2+\frac 1 n,\infty\right)$ for the depolarizing semigroup implies a number of functional inequalities (see \cite[Section 11]{CM20a}). In particular, $(P_t)$ satisfies the exponential decay bound for the Umegaki relative entropy
\begin{equation*}
\tau(P_t(\rho)\log P_t(\rho))\leq e^{-(1+\frac 2 n)t}\tau(\rho\log\rho),\qquad \rho\geq 0,\,\tau(\rho)=1,
\end{equation*}
or equivalently, the modified logarithmic Sobolev inequality
\begin{equation*}
(1+\frac 2 n)\tau(\rho\log\rho)\leq \tau(\L(\rho)\log\rho),\qquad \rho\geq 0,\,\tau(\rho)=1.
\end{equation*}
In the case $n=2$, we obtain the modified logarithmic Sobolev inequality
\begin{equation*}
2\tau(\rho\log\rho)\leq \tau(\L(\rho)\log\rho),\qquad \rho\geq 0,\,\tau(\rho)=1,
\end{equation*}
where the constant $2$ can be seen to be sharp by comparison with the spectral gap (see \cite[Proposition 11.6]{CM20a} for example).
\end{remark}

\emergencystretch=1em

\printbibliography

\end{document}